\documentclass[a4paper, 12pt]{amsart}

\usepackage{amssymb}
\usepackage{amsthm}
\usepackage{amsmath}
\usepackage[mathscr]{eucal}
\usepackage{comment}

\usepackage{longtable}
\usepackage{tabu}
\usepackage{tabularx}
\usepackage{ltablex}

\usepackage[pdftex]{graphicx}

\theoremstyle{plain}
\newtheorem{thm}{Theorem}[section]
\newtheorem{prop}[thm]{Proposition}
\newtheorem{cor}[thm]{Corollary}
\newtheorem{lem}[thm]{Lemma}

\theoremstyle{definition}
\newtheorem{df}{Definition}[section]

\theoremstyle{remark}
\newtheorem{rmk}{Remark}[section]
\newtheorem*{ac}{Acknowledgements}

\newtheorem*{conflict}{Conflict of interest}

\newcommand{\nn}{\mathbb{N}}
\newcommand{\zz}{\mathbb{Z}}
\newcommand{\qq}{\mathbb{Q}}
\newcommand{\rr}{\mathbb{R}}

\newcommand{\gh}{\mathcal{GH}}

\DeclareMathOperator{\card}{Card}

\DeclareMathOperator{\cl}{CL}

\DeclareMathOperator{\met}{Met}

\newcommand{\grsp}{\mathscr{M}}
\DeclareMathOperator{\grdis}{\mathcal{GH}}
\newcommand{\cind}{\nn}

\DeclareMathOperator{\hdis}{\mathcal{H}}

\DeclareMathOperator{\dis}{dis}
\newcommand{\setd}{\mathscr{DO}}
\newcommand{\setud}{\mathscr{UD}}
\newcommand{\setup}{\mathscr{UP}}

\newcommand{\qcube}{\mathbf{Q}}

\DeclareMathOperator{\aaa}{\mathcal{A}}

\DeclareMathOperator{\ccc}{\mathscr{CA}}
\DeclareMathOperator{\iii}{\mathscr{I}}
\newcommand{\dimdim}{\mathcal{D}}

\newcommand{\rela}{\mathscr{R}}
\newcommand{\relacc}{\mathscr{CR}}

\newcommand{\mset}{K}
\newcommand{\msetdis}{k}

\newcommand{\hensu}{\mathcal{W}}

\makeatletter
\renewcommand{\p@enumii}{}
\makeatother

\makeatletter
\@addtoreset{equation}{section}

\makeatother

\begin{document}

\title[Branching geodesics]
{
 Branching geodesics of the Gromov--Hausdorff distance
 }

\thanks{
This paper is published in 
\textit{Analysis and Geometry in Metric Spaces},  
Volume 10,  Issue 1, 2022, 109--128.
(Open Access)}
\author[Yoshito Ishiki]
{Yoshito Ishiki}
\address[Yoshito Ishiki]
{\endgraf
Photonics Control Technology Team
\endgraf
RIKEN Center for Advanced Photonics
\endgraf
2-1 Hirasawa, Wako, Saitama 351-0198, Japan}
\email{yoshito.ishiki@riken.jp}


\begin{abstract}
In this paper,  
we first evaluate topological distributions  of 
the sets of all doubling spaces, all uniformly disconnected spaces, and all uniformly perfect spaces  in the
space of all isometry  classes of compact metric spaces
equipped with the 
 Gromov--Hausdorff distance.  
We then  construct  
branching geodesics of 
the Gromov--Hausdorff distance continuously  parameterized by the Hilbert cube,  
passing through or  avoiding 
 sets of all spaces satisfying  some of  
the three properties shown above, 
and 
 passing through the sets of all infinite-dimensional spaces and the set of all Cantor metric spaces. 
Our construction implies 
that for 
every pair of compact metric spaces, 
there exists a topological embedding 
of  the Hilbert cube  into the Gromov--Hausdorff space
whose image contains the pair. 
From  our results, 
 we observe  that 
 the sets explained above are geodesic spaces
 and infinite-dimensional. 
\end{abstract}

\maketitle

\tableofcontents

\section{Introduction}
In this paper, 
we denote by 
 $\grsp$  
 the  set of all isometry classes of
  non-empty compact metric spaces, and 
 denote by 
 $\grdis$
the Gromov--Hausdorff distance. 
We refer to  
$(\grsp, \grdis)$ as  the 
\emph{Gromov--Hausdorff space}.
We denote by 
$\qcube$ 
the product space 
$[0, 1]^{\nn}$ 
of the countable copies of the unit interval.  
The space 
$\qcube$ 
is called the \emph{Hilbert cube}.

In this paper,  we first evaluate   topological distributions of the sets of all doubling spaces, all uniformly disconnected spaces, and 
all uniformly perfect spaces in 
$(\grsp, \grdis)$, respectively. 
We then show that 
the existence of 
continuum many branching geodesics 
passing through or avoiding 
 sets of all spaces satisfying  some of  
the three properties shown above, 
or passing through  
the sets of all infinite-dimensional spaces 
and the set of all Cantor metric spaces, 
 by  constructing  
 a family of geodesics continuously parametrized by 
the Hilbert cube. 
This construction implies that 
for a given pair of compact metric spaces, 
there exists a topological embedding from 
the Hilbert cube 
into 
$\grsp$
whose image contains 
a  pair of compact metric spaces. 
From  our results, 
 we observe  that 
 the sets explained above are geodesic spaces
 and infinite-dimensional.

Before precisely  stating our results, 
we introduce  basic concepts. 

Let 
$N\in \nn$. 
A metric space 
$(X, d)$ is said to be 
\emph{$N$-doubling} 
if for all  
$x\in X$ 
and 
$r\in (0, \infty)$ 
there exists a subset 
$F$ of $X$ 
satisfying  that 
$B(x, r)\subset \bigcup_{y\in F}B(y, r/2)$ 
and 
$\card(F)\le N$, 
where 
$B(x, r)$ 
is the closed ball centered at 
$x$ 
with radius 
$r$, 
and the symbol 
``$\card$'' 
stands for the cardinality. 
A metric space is said to be  
\emph{doubling} if
 it is 
 $N$-doubling for some $N$.

Let 
$\delta\in (0, \infty)$. 
A metric space 
$(X, d)$ is said to be 
\emph{ $\delta$-uniformly disconnected} 
if 
for every non-constant finite sequence 
$\{z_i\}_{i=1}^N$ 
in 
$X$
we have 
$\delta d(z_1, z_N)\le \max_{1\le i\le N}d(z_i, z_{i+1})$. 
A metric space is said to be 
\emph{uniformly disconnected} if it is 
$\delta$-uniformly disconnected 
for some $\delta \in (0, \infty)$. 
Note that 
a metric space is uniformly disconnected 
if and only if it is bi-Lipschitz embeddable 
into an ultrametric space (see \cite[Proposition 15.7]{DS1997}).

Let 
$c\in (0, 1)$. 
A metric space 
$(X, d)$ 
is said to be 
\emph{$c$-uniformly perfect} 
if
for every 
$x\in X$, 
and 
for every 
$r\in (0, \delta_d(X))$, 
there exists 
$y\in X$ 
with 
$c\cdot r\le d(x, y)\le r$, 
where 
$\delta_{d}(X)$ 
stands for the diameter. 
A metric space is said to be 
\emph{uniformly perfect} if 
it is 
$c$-uniformly perfect for some 
$c\in (0, 1)$.

Let 
$(X, d)$ 
and 
$(Y, e)$ 
be  metric spaces. 
A homeomorphism 
$f:X\to Y$ 
is said to be 
\emph{quasi-symmetric} 
if 
there exists  a homeomorphism 
$\eta:[0,\infty)\to [0,\infty)$ 
such that 
for all 
$x, y, z\in X$ 
and 
for every 
$t\in [0, \infty)$ 
the inequality  
$d(x, y)\le td(x, z)$
implies  the inequality
$e(f(x), f(y))\le \eta(t)e(f(x), f(z))$. 
For example, 
all bi-Lipschitz homeomorphisms  
are quasi-symmetric. 
Note that the doubling property, 
the uniform disconnectedness, 
and 
the uniform perfectness 
are
invariant under quasi-symmetric maps. 
In this paper, 
we denote by 
$\Gamma$ 
the Cantor set. 
David and Semmes \cite{DS1997} 
proved  
that 
if 
a compact metric space is  
doubling,  
uniformly disconnected, 
and 
uniformly perfect,  
then  
 it is quasi-symmetrically equivalent 
 to the  Cantor set 
 $\Gamma$ 
 equipped  with the Euclidean metric 
 (\cite[Proposition 15.11]{DS1997}).

Let 
$X$ 
be a  topological space. 
A subset 
$S$
 of 
 $X$ 
 is said to be \emph{nowhere dense} if 
the complement of the closure of 
$S$ 
is dense in 
$X$. 
A subset 
of 
$X$ 
is said to be  \emph{meager} if 
it
is 
the  union of countable nowhere dense subsets of 
$X$. 
A subset of $X$
is said to be 
\emph{comeager} if its complement is meager. 
A subset  of 
$X$
is said to be 
$F_{\sigma}$ (resp.~$G_{\delta}$) 
if 
it is 
the union  of  countably many  closed subsets of 
$X$ 
(resp.~the intersection of countably many open subsets of 
$X$).  
A subset 
of 
$X$ 
is said to be 
$F_{\sigma\delta}$ 
(resp.~$G_{\delta\sigma}$) 
if 
it
is the intersection   of  countably many  
$F_{\sigma}$ 
subsets of 
$X$ 
(resp.~the union  of countably many 
$G_{\delta}$ subsets of 
$X$).

There are some results on topological distributions 
in the Gromov--Hausdorff space and spaces of metrics. 
Rouyer \cite{Rouyer2011} 
proved that
several properties on metric spaces are generic. 
For example, 
it was  proven
 that the set of  
 all compact metric spaces homeomorphic 
 to the  Cantor set, 
 the set of all 
compact metric spaces 
with zero Hausdorff dimension and 
lower box dimension, 
and the set of all compact metric spaces with infinite 
upper box dimension  are 
 comeager in 
$(\grsp, \grdis)$.

For a metrizable space 
$X$, 
we denote by 
$\met(X)$ 
the space of 
all metrics generating the same topology of 
$X$. 
We consider that 
$\met(X)$ 
is equipped with the supremum metric. 
In \cite{Ishiki2020int} and \cite{ishiki2021dense}, 
the author determined  
 the topological distributions of 
the doubling property, 
the uniform disconnectedness, 
 the uniform perfectness, 
and their negations  in 
$\met(X)$ 
for a suitable space 
$X$. 
For example, 
in \cite{Ishiki2020int}, it was proven that 
the set of all non-doubling metrics 
and 
the set of all non-uniformly doubling metrics 
are
dense 
$G_{\delta}$ 
in 
$\met(X)$ 
for a non-discrete space 
$X$.

Let 
$\setd$, 
$\setud$, 
and 
$\setup$ be  
the sets of all doubling metric spaces, 
 all uniformly disconnected metric spaces, 
and 
all 
uniformly perfect metric spaces in 
$\grsp$, 
respectively.

\begin{thm}\label{thm:typicality}
The sets 
$\setd$, 
$\setud$, 
and 
$\setup$ 
are dense  
$F_{\sigma}$ 
and meager in 
the Gromov--Hausdorff space 
$(\grsp, \grdis)$.  
\end{thm}

To simplify our description,
 the symbols 
 $\mathscr{P}_1$, 
 $\mathscr{P}_2$,  
 and $\mathscr{P}_3$ 
stand for the doubling property, 
the uniform disconnectedness, 
and 
the uniform perfectness, 
respectively. 
Let 
$\mathscr{P}$ 
be a property of metric spaces. 
If a metric space 
$(X, d)$ 
satisfies 
the property 
$\mathscr{P}$, 
then we write 
$T_{\mathscr{P}}(X, d)=1$; 
otherwise, 
$T_{\mathscr{P}}(X, d)=0$.
For a triple 
$(u_1,u_2,u_3)\in \{0,1\}^3$, 
we say 
that 
a metric space 
$(X, d)$ 
is 
\emph{of type $(u_1,u_2,u_3)$} 
if we have
$T_{\mathscr{P}_k}(X, d)=u_k$ 
for all 
$k\in \{1, 2, 3\}$.

A topological space is said to be a 
\emph{Cantor space} 
if 
it is 
homeomorphic to the Cantor set 
$\Gamma$. 
The author \cite{Ishiki2019} 
proved that 
for every 
$(u, v, w)\in \{0, 1\}^{3}$ 
except 
$(1, 1, 1)$, 
the set of all quasi-symmetric equivalence classes of 
Cantor metric spaces of type 
$(u, v, w)$ 
has exactly  continuum many  elements. 
In \cite{ishiki2021dense}, 
the author determined the 
topological distribution of 
the set of all metrics of type 
$(u, v, w)$ 
in 
$\met(\Gamma)$. 
In this paper,  we develop these results in  
the context of 
the Gromov--Hausdorff space.

Let
$\mathscr{Q}_{1}=\setd$, 
$\mathscr{Q}_{2}=\setud$, 
and 
$\mathscr{Q}_{3}=\setup$.
For 
$k\in \{1, 2, 3\}$, 
and for 
$u\in \{0, 1, 2\}$, 
we define 
\[
\mathscr{E}_{k}(u)=
\begin{cases}
\grsp \setminus \mathscr{Q}_{k} & \text{if $u=0$;}\\
\mathscr{Q}_{k} & \text{if $u=1$;}\\
\grsp & \text{if $u=2$}. 
\end{cases}
\]
For 
$(u, v, w)\in \{0, 1, 2\}^{3}$,  
we  also define

\[
\mathscr{X}(u, v, w)=
 \mathscr{E}_{1}(u)\cap \mathscr{E}_{2}(v)\cap \mathscr{E}_{3}(w). 
\]
The next theorem is an analogue  of 
\cite[Theorem 1.6]{ishiki2021dense}. 
\begin{thm}\label{thm:topdisuvw}

Let $(u, v, w)\in \{0, 1, 2\}^{3}$. Then 
the following statements hold true. 
\begin{enumerate}
\item 
If 
$\{u, v, w\}\setminus \{2\}=\{1\}$, 
then 
the set 
$\mathscr{X}(u, v, w)$ is dense $F_{\sigma}$. 
\item 
If 
$\{u, v, w\}\setminus \{2\}=\{0\}$, 
then 
the set 
$\mathscr{X}(u, v, w)$ 
is dense 
$G_{\delta}$. 
\item 
If 
$\{u, v, w\}\setminus \{2\}=\{0, 1\}$, 
then
the set 
$\mathscr{X}(u, v, w)$ 
is dense 
$F_{\sigma\delta}$ 
and 
$G_{\delta\sigma}$. 
\end{enumerate}
\end{thm}

Let 
$(X, d)$ 
be a metric space, and $a, b\in \rr$. 
A continuous map  
$\gamma:[a, b]\to X$ 
is said to be 
a 
\emph{curve}. 
For $x, y\in X$, 
a curve 
$\gamma:[0, 1]\to X$ 
is said to be  
a \emph{geodesic from $x$ to $y$} if 
$\gamma(0)=x$ and $\gamma(1)=y$, and 
for all 
$s, t\in [0, 1]$ 
we have 
\[
d(\gamma(s), \gamma(t))=|s-t|\cdot d(x, y). 
\]
Note that if there exists a curve whose length is 
$d(x, y)$, then there exists a geodesic from $x$
to $y$ (see \cite[Chapter 2]{BBI}). 
A metric space is said to be a 
\emph{geodesic space} if 
for all two points, 
there exists a geodesic connecting  them.

Ivanov, Nikolaeva, and Tuzhilin  \cite{INT2016}
 proved that 
 $(\grsp, \grdis)$ 
 is a 
 geodesic 
space 
by showing 
the existence of the mid-point of all two points of 
$\grsp$. 
Klibus \cite{K2018} proved that the closed ball in the Gromov--Hausdorff space centered at the one-point metric space
is  a geodesic space. 
Chowdhury and M\'{e}moli \cite{CM2018} 
constructed an explicit  geodesic in 
$(\grsp, \grdis)$ 
 using 
an optimal closed correspondence 
(see also  \cite{IIT2016}). 
They also 
 showed that 
 $(\grsp, \grdis)$ 
 permits branching geodesics by  constructing 
branching geodesics from the one-point metric space. 
M\'{e}moli and Wan \cite{memoli2021characterization} 
showed that every Gromov--Hausdorff geodesic is 
realizable as a geodesic in the Hausdorff hyperspace of some metric space. 
For every pair of two distinct compact metric spaces, they also constructed countably many geodesics connecting that metric spaces. 
As a development of their results on branching geodesics,  in this paper, for all pairs of compact metric spaces in 
$\grsp$,  we construct branching geodesics connecting 
them continuously parametrized by $\qcube$.

\begin{df}\label{df:bbg}
Let 
$A$ 
be a closed subset of 
$[0, 1]$ with $\{0, 1\}\subset A$. 
Let 
$(X, d), (Y, e)\in \grsp$. 
We say that a continuous map 
$F:[0, 1]\times \qcube \to \grsp$ 
is 
an 
\emph{$A$-branching 
bunch of geodesics from $(X, d)$ to $(Y, e)$} 
if the following are satisfied:
\begin{enumerate}
\item 
for every  
$q\in \qcube $, 
we have 
$F(0, q)=(X, d)$ 
and 
$F(1, q)=(Y, e)$;\label{item:end}
\item 
for every 
$s\in A$ 
and  for all 
$q, r\in \qcube$
we have 
$F(s, q)=F(s, r)$;\label{item:A}
\item 
for each 
$q\in \qcube $, 
the map 
$F_{q}: [0, 1]\to \grsp$ 
defined by 
$F_{q}(s)=F(s,  q)$ 
is 
a geodesic from  
$(X, d)$ to 
$(Y, e)$;\label{item:Fq}
\item 
for all  
$(s, q),  (t, r)\in ([0, 1]\setminus A)\times \qcube$ 
with 
$(s, q)\neq (t, r)$, 
we have 
$F(s, q)\neq F(t, r)$. 
\label{item:noniso}
\end{enumerate}
\end{df}

\begin{figure}[h]
\centering 
\includegraphics[width=\columnwidth]{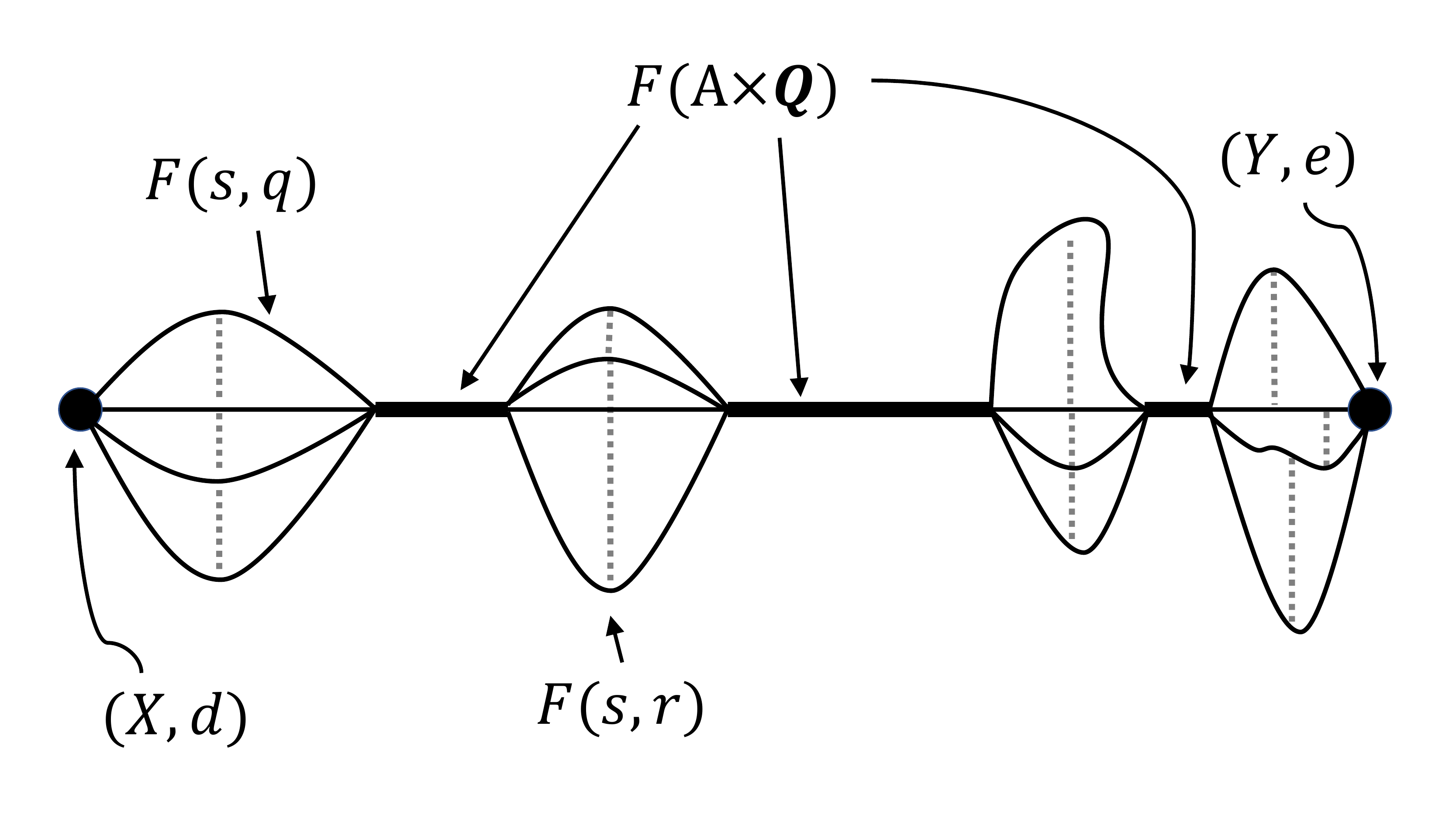}
\caption{$A$-branching bunch of geodesics}
\end{figure}

\begin{thm}\label{thm:q-emb}
Let 
$(u, v, w)\in \{0, 1, 2\}^{3}$. 
Let 
$(X, d), (Y, e)\in \mathscr{X}(u, v, w)$. 
Let 
$A$ 
be a closed subset of 
$[0, 1]$ with $\{0, 1\}\subset A$. 
Then there exists 
an $A$-branching bunch of  geodesics 
$F: [0, 1]\times \qcube \to \mathscr{X}(u, v, w)$  
from 
$(X, d)$ to 
$(Y, e)$. 
\end{thm}

We say that a map
$\dimdim:\grsp\to [0, \infty]$ 
is a \emph{dimensional function} if
\begin{enumerate}
\item 
 for every  
 $(X, d)\in \grsp$, 
 and for every closed subset 
 $A$ of $X$, 
we have 
$\dimdim(A, d)\le \dimdim(X, d)$;
\item 
for all $\epsilon\in (0, \infty)$, 
there exists a compact metric space 
$(Y, e)$ with 
$\dimdim(Y, e)=\infty$ and $\delta_{e}(Y)\le \epsilon$.\label{item:dimfunc2}
\end{enumerate}
For example, 
the covering  dimension (topological dimension) 
$\dim$, 
the Hausdorff dimension 
$\dim_{H}$, 
the packing dimension 
$\dim_{P}$, 
the lower box dimension 
$\underline{\dim_{B}}$, 
the upper box dimension 
$\overline{\dim_{B}}$,
and 
the Assouad dimension 
$\dim_{A}$
are dimensional functions. 

For a dimensional function 
$\dimdim$, 
we denote by  
$\iii(\dimdim)$  
the set of all compact metric spaces 
$(X, d)\in \grsp$ 
with 
$\dimdim(X, d)=\infty$. 
Note that 
$\iii(\dimdim)\neq \emptyset$.

\begin{thm}\label{thm:q-embinfinite}
Let 
$\dimdim$ 
be a dimensional function. 
Let 
$(X, d), (Y, e)\in \iii(\dimdim)$, 
and  
$A$  
a closed subset of 
$[0, 1]$ with $\{0, 1\}\subset A$. 
Then there exists 
an $A$-branching bunch of  geodesics 
$F: [0, 1]\times \qcube \to \iii(\dimdim)$
  from 
  $(X, d)$ to 
  $(Y, e)$. 
\end{thm}

Let $\ccc$ denote the set of all Cantor metric space  in $\grsp$. 
Note that $\ccc$ is comeger in $\grsp$ (see \cite{Rouyer2011}).
By the same method of the proof of Theorem \ref{thm:q-emb}, we obtain:
\begin{thm}\label{thm:Cantoremb}
Let 
$(X, d), (Y, e)\in \ccc$, 
and  
$A$  
a closed subset of 
$[0, 1]$ with $\{0, 1\}\subset A$. 
Then there exists 
an $A$-branching bunch of  geodesics 
$F: [0, 1]\times \qcube \to \ccc$
  from 
  $(X, d)$ to 
  $(Y, e)$. 
\end{thm}

From  Theorems \ref{thm:q-emb}, 
\ref{thm:q-embinfinite}, and 
\ref{thm:Cantoremb} 
we obtain the following four statements. 
Theorem \ref{thm:branch} is  an improvement  of 
\cite[Theorem 5.13]{memoli2021characterization}. 
\begin{thm}\label{thm:branch}
Let 
$\mathscr{S}$ 
be any one of 
$\mathscr{X}(u, v, w)$ 
for
 some 
 $(u, v, w)\in \{0, 1, 2\}^{3}$ 
or 
$\iii(\dimdim)$ 
for some dimensional function 
$\dimdim$ 
or 
$\ccc$. 
Then 
for  all 
$(X, d), (Y, e)\in \mathscr{S}$ satisfying 
$\grdis((X, d), (Y, e))>0$, 
there are  exact continuum many  geodesics from 
$(X, d)$ to  
$(Y, e)$ passing through  
$\mathscr{S}$. 
\end{thm}

In this paper, 
a metric space  is said to be 
\emph{infinite-dimensional}
if its 
covering dimension is infinite. 
\begin{thm}\label{thm:topemb}
Let 
$\mathscr{S}$ 
be any one of 
$\mathscr{X}(u, v, w)$ 
for some 
$(u, v, w)\in \{0, 1, 2\}^{3}$ 
or 
$\iii(\dimdim)$ 
for some dimensional function 
$\dimdim$ 
or  $\ccc$.
Then  for  all 
$(X, d), (Y, e)\in \mathscr{S}$, 
there exists a topological embedding from 
the Hilbert cube 
$\qcube$ into 
$\mathscr{S}$
 whose image contains 
$(X, d)$ and 
$(Y, e)$. 
In particular, 
  $\mathscr{S}$ 
is infinite-dimensional.  
\end{thm}

The following theorem  states that  the sets 
$\iii(\dimdim)$ 
and 
$\mathscr{X}(u, v, w)$ and $\ccc$ 
are everywhere 
infinite-dimensional. 
\begin{thm}\label{thm:locinfinite}
Let 
$\mathscr{S}$ 
be any one of  
$\mathscr{X}(u, v, w)$ 
for some 
$(u, v, w)\in \{0, 1, 2\}^{3}$ 
or 
$\iii(\dimdim)$ 
for some 
dimensional function 
$\dimdim$
or $\ccc$. 
Then 
 for every non-empty open subset 
 $O$ 
 of 
$\grsp$, 
the set 
$\mathscr{S}\cap O$ 
is 
infinite-dimensional. 
\end{thm}

Since all separable  metrizable spaces are topologically 
embeddable into $\qcube$, 
we have:
\begin{cor}
If 
 $\mathscr{S}$ 
 is  any one of  
$\mathscr{X}(u, v, w)$ 
for some 
$(u, v, w)\in \{0, 1, 2\}^{3}$, 
or 
$\iii(\dimdim)$ 
for some dimensional function 
$\dimdim$ or 
$\ccc$, 
then 
 every separable  metrizable space 
 $X$ 
 can be topologically embeddable into 
$\mathscr{S}$. 
\end{cor}
\begin{rmk}
It is open whether 
all separable (or compact) metric spaces 
are isometrically embeddable into 
$\grsp$. 
On the other hand, Wan \cite{Wan2021} 
proved that
all separable ultrametric spaces 
are isometrically embeddable into 
 the Gromov--Hausdorff ultrametric space.  
\end{rmk}


The organization of this paper is as follows:
In Section \ref{sec:pre}, 
we prepare and explain 
basic concepts and  statements on metric spaces. 
In Section \ref{sec:topdis}, 
we prove Theorems \ref{thm:typicality} and 
\ref{thm:topdisuvw}. 
In Section \ref{sec:const}, 
we introduce  specific versions of  telescope spaces 
and sequentially metrized Cantor spaces introduced  
in \cite{Ishiki2019}. 
Using  telescope spaces, 
we also construct a family of compact metric spaces 
continuously parameterized  by 
$\qcube$, 
which 
are not isometric to each other. 
In Section \ref{sec:geo}, 
we first prove Theorems 
\ref{thm:q-emb}, 
\ref{thm:q-embinfinite}, and 
\ref{thm:Cantoremb}. 
As its applications, 
we next prove 
Theorems 
\ref{thm:branch}, 
\ref{thm:topemb}, and
\ref{thm:locinfinite}. 
In Section \ref{sec:table},
for the convenience for the readers, 
 we 
exhibit a table of 
symbols. 


\section{Preliminaries}\label{sec:pre}
In this section, 
we prepare 
and explain 
the
basic concepts 
and  
statements 
on metric spaces. 
\subsection{Generalities}
In this paper,  
we denote by 
$\cind$ 
the set of all positive 
integers. 
The symbol 
$\lor$ 
stands for 
the maximal operator of 
$\rr$. 
Let 
$X$ 
be a  set. 
A metric 
$d$ 
on 
$X$ 
is said to be an 
\emph{ultrametric} 
if 
for all 
$x, y, z\in X$ 
the metric 
$d$ 
satisfies 
$d(x, y)\le d(x, z)\lor d(z, y)$. 
In this paper, 
for a metric space 
$(X, d)$, 
and for a subset 
$A$ 
of 
$X$, 
we represent 
the restricted metric 
$d|_{A^2}$ 
as 
  the same symbol 
$d$ 
as  the ambient metric 
$d$.  
For a subset 
$A$ 
of 
$X$, 
we denote by 
$\delta_d(A)$ 
the diameter of 
$A$, 
and 
we define  
$\alpha_d(A)=\inf\{\, d(x, y)\mid x\neq y,\ x, y\in A\, \}$.

The following two lemmas are used to prove our results. 
\begin{lem}\label{lem:lip}
Let 
$L\in (0, \infty)$, 
and  
$A$ 
be a closed subset of 
$[0, 1]$. 
Then there exists an 
$L$-Lipschitz function
$\zeta: [0, 1]\to [0, \infty)$ 
with 
$\zeta^{-1}(0)=A$. 
\end{lem}
\begin{proof}
For 
$x\in [0, 1]$, 
let  
$\xi(x)$ be the distance between 
$x$ 
and 
$A$. 
Then the function 
$L\cdot \xi:[0, 1]\to [0, \infty)$ 
is a desired one. 
\end{proof}

\begin{lem}\label{lem:max}
Let 
$x, y, u, v\in \rr$. 
Then, 
we have 
\[
|x\lor y- u\lor v|\le |x-u|\lor |y-v|. 
\]
\end{lem}
\begin{proof}
We only need to  consider the 
 case of 
 $x\lor y=x$ 
 and 
 $u\lor v=v$. 
By 
$u\le v$, 
we obtain 
$x-v\le x-u\le |x-u|$. 
By 
$y\le x$, 
we obtain 
$v-x\le v-y\le |y-v|$. 
These imply the inequality. 
\end{proof}

\subsection{Quasi-symmetrically invariant properties}
For a metric space $(X, d)$, 
the Assouad dimension 
$\dim_{A}(X, d)$ of $(X, d)$ 
is 
defined by the infimum of all  
$\beta\in (0, \infty)$ 
for which 
 there exists 
 $C\in (0, \infty)$ 
 such that 
 for every finite subset 
 $A$ 
 of 
 $X$ 
we have 
$\card(A)\le C\cdot(\delta_d(A)/\alpha_d(A))^{\beta}$. 
By the definition of the doubling property, we obtain the following two lemmas. 
\begin{lem}\label{lem:doublingprop}
A metric space 
$(X, d)$ 
is doubling 
if and only if
there exist 
$\beta\in (0, \infty)$ 
and 
$C\in [1, \infty)$ 
such that 
for every finite subset 
$A$ 
of 
$X$ 
we have 
$\card(A)\le C\cdot(\delta_d(A)/\alpha_d(A))^{\beta}$. 
\end{lem}

\begin{lem}\label{lem:subAssouad}
Let 
$(X, d)$ be a metric space,  and 
let $A$ be a subset of $X$. 
Then,  we have  $\dim_{A}(A, d)\le \dim_{A}(X, d)$. 
\end{lem}

Note that the doubling property is equivalent to 
the finiteness of the Assouad dimension.

By definitions of ultrametric spaces, we obtain:
\begin{lem}\label{lem:ultraud}
Every  ultrametric space is 
$\delta$-uniformly disconnected for all 
$\delta\in (0, 1)$. 
\end{lem}

For two metric spaces $
(X, d)$ 
and 
$(Y, e)$, 
we denote by 
$d\times_{\infty} e$  
the 
$\ell^{\infty}$-product metric  
defined by 
\[
(d\times_{\infty} e)((x, y), (u, v))=
d(x, u) \lor e(y, v). 
\]
Note that 
$d\times_{\infty} e$ 
generates the product topology of 
$X\times Y$.

In this paper, 
we sometimes use the disjoint union 
$\coprod_{i\in I} X_{i}$ 
of a non-disjoint family 
$\{X_{i}\}_{i\in I}$. 
 Whenever we consider the disjoint  union  
 $\coprod_{i\in I}X_{i}$ of 
 a family 
 $\{X_{i}\}_{i\in I}$ 
 of sets 
 (this family is not necessarily disjoint), 
we  identify the family 
$\{X_{i}\}_{i\in I}$ 
with  its disjoint copy unless otherwise stated. 
If each 
$X_{i}$ 
is a topological space, 
we consider that 
$\coprod_{i\in I}X_{i}$ 
is equipped with 
the direct sum topology. 

Next we review the  basic statements of 
the doubling property, 
the uniform disconnectedness, 
and 
the uniform perfectness. 

The following is presented in \cite[Corollary 10.1.2]{fraser2020assouad}. 
\begin{lem}\label{lem:finAssouad}
Let 
$(X, d)$ 
and 
$(Y, e)$ 
be 
metric spaces. 
Then, 
\[
\dim_{A}(X\times Y, d\times_{\infty} e)\le \dim_{A}(X, d)+\dim_{A}(Y, e).
\] 
\end{lem}

By Lemmas 2.13 and 2.14 in 
\cite{Ishiki2019}, we obtain the following lemma (see also \cite[Remark 2.16]{Ishiki2019}):
\begin{lem}\label{lem:qsprod}
Let 
$(X, d)$ 
and 
$(Y, e)$ 
be 
metric spaces. 
Then for all 
$k\in \{1, 2\}$, 
we obtain 
$T_{\mathscr{P}_{k}}(X\times Y, d\times_{\infty}e)
=T_{\mathscr{P}_{k}}(X, d)
\land T_{\mathscr{P}_{k}}(Y, e)$, 
where 
$\land$ 
is the minimum operator. 
\end{lem}

By \cite[Lemma 5.14]{Ishiki2019},  we obtain:
\begin{lem}\label{lem:qssqcup}
Let 
$(X, d)$ 
and 
$(Y, e)$ 
be 
metric spaces. 
Let 
$h\in \met(X\sqcup Y)$ 
with 
$h|_{X^{2}}=d$ and $h|_{Y^{2}}=e$. 
Then for all 
$k\in \{1, 2, 3\}$, 
we obtain  
$T_{\mathscr{P}_{k}}(X\sqcup Y, h)
=T_{\mathscr{P}_{k}}(X, d)
\land T_{\mathscr{P}_{k}}(Y, e)$. 

\end{lem}

The following is identical with 
\cite[Lemma 2.15]{Ishiki2019}. 
\begin{lem}\label{lem:produp}
Let 
$(X, d)$ 
and 
$(Y, e)$  
be metric spaces. 
If either of the two is uniformly perfect, 
then 
so is 
$(X\times Y, d\times_{\infty}e)$. 
\end{lem}
By the definition of the uniform perfectness, 
we obtain: 
\begin{lem}\label{lem:unionup}
Let 
$(X, d)$ 
be a metric space, 
and 
$A$, 
$B$ 
be subsets of 
$X$. 
If 
$(A, d)$ 
and 
$(B, d)$ are uniformly perfect, 
then so is 
$(A\cup B, d)$. 
\end{lem}

We refer the  readers to 
\cite{H2001} for 
the details of the following:
\begin{prop}\label{prop:qsinvariant3}
The doubling property, 
the uniform disconnectedness,
and the uniform perfectness 
are invariant under quasi-symmetric maps. 
\end{prop}
By the definitions of the doubling property, 
and 
the uniformly disconnectedness, we obtain:
\begin{prop}\label{prop:hered}
Let $(X, d)$ be a metric space, 
and $A$ be a subset of $X$. 
If $(X, d)$ is doubling (resp.~uniformly disconnected), 
then so is $(A, d)$. 
\end{prop}

For a property 
$\mathscr{P}$ 
of metric spaces,  
and for a metric space 
$(X,d)$ 
we define 
$\mathcal{S}_{\mathscr{P}}(X,d)$ 
as the set of all points in 
$X$ 
of which no neighborhoods 
satisfy 
$\mathscr{P}$ (see \cite[Definition 1.3]{Ishiki2019}). 
By Proposition \ref{prop:qsinvariant3}, 
we obtain:
\begin{lem}\label{lem:sop}
Let 
$k\in \{1, 2, 3\}$. 
If metric spaces 
$(X, d)$ 
and 
$(Y, e)$ 
are quasi-symmetric equivalent to  each other, 
then so are 
$\mathcal{S}_{\mathscr{P}_{k}}(X,d)$ 
and 
$\mathcal{S}_{\mathscr{P}_{k}}(Y,e)$. 
In particular,  
$\mathcal{S}_{\mathscr{P}_{k}}(X,d)$ 
is 
an isometric invariant. 
\end{lem}
For a triple 
$(u,v,w)\in \{0,1\}^3$ 
except 
$(1, 1, 1)$, 
we say that a metric space 
$(X,d)$ 
is of  
 \emph{totally exotic type $(u,v,w)$} 
if 
$(X,d)$ is 
of  type $(u,v,w)$, 
and if 
$\mathcal{S}_{\mathscr{P}_{k}}(X,d)=X$ 
holds for all 
$k\in \{1, 2, 3\}$ 
with 
$T_{\mathscr{P}_{k}}(X,d)=0$. 
This concept was introduced  in the author's paper \cite{Ishiki2019}, 
and the existence of such spaces was proven in 
\cite[Theorem 1.7]{Ishiki2019}. 
\begin{thm}\label{thm:totexo}
For every 
$(u,v,w)\in \{0,1\}^3$ 
except 
$(1,1,1)$, 
there exists 
a Cantor metric space of  totally exotic type 
$(u,v,w)$. 
\end{thm}
The definition of totally exotic types and Proposition \ref{prop:hered} imply:
\begin{lem}\label{lem:timestotexo}
Let 
$(u,v,w)\in \{0,1\}^3$ 
except 
$(1,1,1)$. 
Let 
$(X, d)$ 
be a  metric space of  totally exotic type 
$(u, v, w)$. 
Let 
$(Y, e)$ 
be a metric space. 
Then for all 
$k\in \{1, 2\}$ with $T_{\mathscr{P}_{k}}(X, d)=0$, 
we obtain 
$\mathcal{S}_{\mathscr{P}_{k}}(X\times Y, d\times_{\infty}e)
=X\times Y$. 
\end{lem}

\subsection{The Gromov--Hausdorff distance}\label{subsec:gh}
For a metric space 
$(Z, h)$,  
and  for subsets 
$A$,  $B$ 
of 
$Z$, 
we denote by 
$\hdis(A,B;Z, h)$ 
the Hausdorff distance of 
$A$ 
and 
$B$ 
in 
$Z$. 
For  metric spaces 
$X$ 
and 
$Y$, 
the 
\emph{Gromov--Hausdorff distance} 
$\grdis((X, d),(Y, e))$ between 
$X$ 
and 
$Y$ 
is defined as 
the infimum of  all values  
$\hdis(i(X), j(Y); Z, h)$, 
where 
$(Z, h)$ 
is a metric space, 
 and 
$i: X\to Z$ 
and 
$j: Y\to Z$ 
are isometric embeddings. 
For  $\epsilon\in (0, \infty)$, 
and for metric spaces 
$X$ 
and 
$Y$, 
a pair 
$(f, g)$ 
with 
$f:X\to Y$ 
and 
$g:Y\to X$ 
is said to be 
an \emph{$\epsilon$-approximation} between 
$(X, d)$ 
and 
$(Y, e)$
if the following conditions hold: 
\begin{enumerate}
\item for all 
$x, y\in X$, 
we have 
$|d(x,y)-e(f(x),f(y))|<\epsilon$; 
\item for all 
$x,y\in Y$, 
we have 
$|e(x,y)-d(g(x), g(y))|<\epsilon$; 
\item for each 
$x\in X$ 
and for each 
$y\in Y$, 
we have 
$d(g\circ f(x),x)<\epsilon$ 
and 
$e(f\circ g(x), x)<\epsilon$. 
\end{enumerate}
Remark that if there exists a map  $f: (X, d)\to (Y, e)$ satisfying (1) and $\bigcup_{y\in f(X)}B(y, \epsilon)=Y$, then there exists $g:(Y, e)\to (X, d)$ such that $(f, g)$ is a $3\epsilon$-approximation between $(X, d)$ and $(Y, e)$ (see \cite{jansen2017notes}). 

The proof of the next lemma 
is presented 
in 
\cite{BBI} and \cite{jansen2017notes}. 
\begin{lem}\label{lem:approx}
Let 
$\{(X_{i}, d_{i})\}_{i\in \cind}$  
be 
a  sequence 
of 
compact metric spaces, 
and 
$(X, d)$ 
be a compact metric space. 
Then 
$\grdis((X_{i}, d_{i}), (X, d))\to 0$ 
as 
$i\to \infty$
if and only  if 
there exists a sequence 
$\{\epsilon_{i}\}_{i\in \cind}$ 
in 
$(0, \infty)$ 
converging to 
$0$,  
and for each 
$i\in \cind$ 
there exists 
an  
$\epsilon_{i}$-approximation 
between 
$(X_{i}, d_{i})$ 
and 
 $(X, d)$. 
\end{lem}

Let 
$(X, d)$ 
and 
$(Y, e)$ 
be compact metric spaces. 
We say that  a subset 
$R$ 
of 
$X\times Y$ 
is a \emph{correspondence} if
$\pi_X(R)=X$ 
and 
$\pi_Y(R)=Y$, 
where
$\pi_{X}$ 
and 
$\pi_{Y}$ 
are projections into 
$X$ 
and 
$Y$,  
respectively. 
We denote by 
$\rela(X,Y)$ 
(resp.~$\relacc(X, d,Y, e)$) 
the set of all correspondences 
(resp.~closed correspondences in 
$X\times Y$) 
of 
$X$ 
and  
$Y$. 
For 
$R\in \rela(X,Y)$, 
we define 
the 
\emph{distortion $\dis(R)$ of $R$} 
by 
\[
\dis(R)=\sup_{(x,y), (u,v)\in R}|d(x,u)-e(y,v)|.
\]

The proof of the next lemma 
is presented 
 in  \cite{BBI}. 
\begin{lem}\label{lem:ghestimate}
For all compact metric spaces 
$(X, d)$ 
and 
$(Y, e)$, 
we obtain 
\[
\grdis((X, d), (Y, e))=\frac{1}{2}\inf_{R\in \rela(X,Y)}\dis(R). 
\] 

\end{lem}

Let  
$(X, d)$ 
and 
$(Y, e)$ 
be  compact metric spaces. 
We denote by 
$\relacc_{opt}(X, d, Y, e)$ 
the 
set of all 
$G\in \relacc(X, d, Y, e)$ 
satisfying that 
$\dis(G)=\inf_{R\in \rela(X,Y)}\dis(R)$.
An element of 
$\relacc_{opt}(X, d, Y, e)$ 
is 
said to be 
\emph{optimal}. 
The proof of the next lemma is presented 
in \cite{CM2018} and \cite{IIT2016}. 
\begin{lem}\label{lem:opt}
If 
$(X, d)$ 
and 
$(Y, e)$ 
are compact metric spaces, 
then 
$\relacc_{opt}(X, d, Y, e)\neq \emptyset$.  
\end{lem}

For a set $X$, we define  
$\Delta_{X}\in \rela(X, X)$ 
by 
$\Delta_{X}=\{\, (x, x)\mid x\in X\, \}$, 
and we call it  the 
\emph{trivial correspondence of $X$}.

\subsection{Amalgamation of metrics}
Since the following lemma seems to be classical, 
we omit the proof. 
For 
$n\in \cind$, 
we put 
$\widehat{n}=\{1, 2, \dots, n\}$, 
and 
we consider that 
$\widehat{n}$
 is always  equipped with the discrete topology. 
\begin{lem}\label{lem:amaldiam}
Let 
$(X_{i}, d_{i})$ 
be a sequence of metric spaces. 
Let 
$e\in \met(\widehat{n})$. 
Assuming that 
$\delta_{d_{i}}(X_{i})\le 
\alpha_{e}\left(\widehat{n}\right)$, 
we 
define a symmetric function 
$D: (\coprod_{i=1}^{n}X_{i})^{2}\to [0, \infty)$
by 
\[
D(x, y)=
\begin{cases}
d_{i}(x, y) & \text{if $x, y\in X_{i}$}\\
e(i, j) & \text{if $x\in X_{i}$ and $y\in X_{j}$ with $i\neq j$}
\end{cases}
\]
Then 
$D\in \met(\coprod_{i=1}^{n}X_{i})$. 
\end{lem}
Let $(X, d)$ be a metric space and $\epsilon\in (0, \infty)$. A finite  subset $S$ of $X$ is an \emph{$\epsilon$-net} if $\alpha_{d}(S)\ge \epsilon$ and $\bigcup_{x\in S}B(x, \epsilon)=X$. 
\begin{lem}\label{lem:denseuvw}
The following holds true:
\begin{enumerate}
\item 
For every 
$(u, v, w)\in \{0, 1, 2\}^{3}$, 
the set 
$\mathscr{X}(u, v, w)$ 
is dense in 
$\grsp$.\label{item:uvwdense}
\item 
For every dimensional function $\dimdim$, 
the set $\iii(\dimdim)$ is dense in $\grsp$.\label{item:iddense}
\item 
The set $\ccc$ is dense in $\grsp$.\label{item:cccdense}
\end{enumerate}
\end{lem}
\begin{proof}We first deal with the statement 
(\ref{item:uvwdense}). 
We only  need to  consider 
 the cases of 
 $(u, v, w)\in \{0, 1\}^{3}$. 
Take 
$(X, d)\in \grsp$, 
and 
$\epsilon \in (0, \infty)$.
 Let 
 $A=\{a_{i}\}_{i=1}^{n}$ 
 be an 
 $\epsilon$-net 
 of 
 $(X, d)$. 
 Define 
 $e\in \met(\widehat{n})$ 
 by 
 $e(i, j)=d(a_{i}, a_{j})$. 
 Put 
 $\eta=\min\{\epsilon, \alpha_{e}(\widehat{n})\}$. 
 Take a Cantor metric space 
 $(M, h)$ 
 of type 
 $(u, v, w)$ 
 (see \cite[Theorem 1.2]{Ishiki2019}). 
By  replacing with 
 $(M, (\eta/\delta_{h}(M))h)$ 
 if necessary,
  we may assume that 
  $\delta_{h}(M)\le \eta$. 
For each 
$i\in \{1, \dots, n\}$, 
 let 
 $(X_{i}, d_{i})$ 
 be an isometric copy of 
 $(M, h)$. 
Let
  $Z=\coprod_{i=1}^{n}X_{i}$. 
 Let 
 $D$ 
 be a metric 
stated in 
Lemma \ref{lem:amaldiam} 
induced from  
$e$ 
and 
 $\{(X_{i}, d_{i})\}_{i=1}^{n}$.
Then, 
by Lemma \ref{lem:qssqcup},  
the space 
$(Z, D)$ 
is 
of type 
$(u, v, w)$, 
and we obtain 
\[
\grdis((X, d), (Z, D))\le \grdis((X, d), (A, d))+\grdis((A, d), (Z, D))\le 2\epsilon. 
\]
 Thus, 
 $\mathscr{X}(u, v, w)$ 
 is 
 dense in 
 $\grsp$. 
 
 The statements (\ref{item:iddense}) and 
 (\ref{item:cccdense}) can be proven in  the same way as 
 the statement (\ref{item:uvwdense}).  For the statement  (\ref{item:iddense}), we  use the condition (\ref{item:dimfunc2}) in the definition of dimensional functions. 
\end{proof}


\section{Topological distributions}\label{sec:topdis}
In this section, 
we prove Theorems \ref{thm:typicality}
and \ref{thm:topdisuvw}.

\begin{lem}\label{lem:setd}
The set 
$\setd$ 
is  
$F_{\sigma}$ 
and  meager  in 
$\grsp$. 
\end{lem}
\begin{proof}
Since 
$\grsp\setminus \setd$ 
is dense (see Lemma \ref{lem:denseuvw}), 
it suffices to show that 
$\setd$ 
is 
$F_{\sigma}$. 
For 
$C, \beta\in (0, \infty)$, 
let 
$\mathscr{S}(C, \beta)$ 
be the set of all 
$(X, d)$ 
such that 
for every finite subset of 
$X$ 
we have 
\[
\card(A)\le 
C\cdot \left(\frac{\delta_{d}(A)}{\alpha_{d}(A)}\right)^{\beta}. 
\]
We now prove that 
 each 
 $\mathscr{S}(C, \beta)$ 
 is closed 
in 
$\grsp$. 
Take a convergent sequence  
$\{(X_{i},  d_{i})\}_{i\in \cind}$ 
in 
$\mathscr{S}(C, \beta)$, and 
let 
$(X, d)$ 
be its limit space. 
Then,  
by Lemma \ref{lem:approx}, 
 there exist 
 a positive sequence  
 $\{\epsilon_{i}\}_{i\in \cind}$ 
 converging  to 
 $0$, 
and  sequences 
$\{f_{i}: (X_{i}, d_{i})\to (X, d)\}_{i\in \cind}$ 
and 
$\{g_{i}: (X, d)\to (X_{i}, d_{i})\}_{i\in \cind}$
such that for each 
$i\in \cind$ 
the pair 
$(f_{i}, g_{i})$ 
is an 
$\epsilon_{i}$-approximation. 
 Take an arbitrary  finite subset 
 $A$ 
 of 
 $X$. 
Take a sufficiently large 
$i\in \cind$, 
then 
$g_{i}: A\to g_{i}(A)$ 
is bijective. 
Thus, 
\begin{align*}
\card(A)=\card(g_n(A))
\le C\cdot 
\left(
\frac{\delta_{d_{i}}(g_{i}(A))}{\alpha_{d_{i}}(g_{i}(A))}
\right)^{\beta}
\le C\cdot \left(
\frac{\delta_{d}(A)+\epsilon_{i}}{\alpha_{d}(A)-\epsilon_{i}}
\right)^{\beta}. 
\end{align*}
By letting 
$i\to \infty$, 
we obtain 
\[
\card(A)\le  C\cdot \left(\frac{\delta_{d}(A)}{\alpha_{d}(A)}\right)^{\beta}. 
\]
Then 
$(X, d)\in \mathscr{S}(C, \beta)$, 
and hence 
$\mathscr{S}(C, \beta)$ 
is closed in 
$\grsp$.

By Lemma \ref{lem:doublingprop}, 
we obtain 
\[
\setd=\bigcup_{C, \beta\in \qq_{>0}}\mathscr{S}(C, \beta). 
\]
Therefore, 
we conclude that 
$\setd$ 
is 
$F_{\sigma}$ 
in 
$\grsp$. 
\end{proof}

\begin{lem}\label{lem:setud}
The set 
$\setud$ 
is 
$F_{\sigma}$ and meager 
in 
$\grsp$. 
\end{lem}
\begin{proof}
Since 
$\grsp\setminus \setud$ 
is dense 
(see Lemma \ref{lem:denseuvw}), 
it suffices to show that 
$\setud$ 
is 
$F_{\sigma}$. 
For
 $\delta\in (0, 1)$, 
we denote by 
$\mathscr{S}(\delta)$ 
the 
set of all 
$\delta$-uniformly disconnected compact metric spaces. 
We now prove that each 
$\mathscr{S}(\delta)$ 
is closed in 
$\grsp$. 
Take a convergent sequence 
$\{(X_{i}, d_{i})\}_{i\in \cind}$ 
in 
$\mathscr{S}(\delta)$, 
and let $(X, d)$ be the its limit compact metric  space. 
Then, 
by Lemma \ref{lem:approx}, 
 there exist 
 a positive  sequence  
 $\{\epsilon_{i}\}_{i\in \cind}$ 
 converging  to 
 $0$, 
and  sequences 
$\{f_{i}: (X_{i}, d_{i})\to (X, d)\}_{i\in \cind}$ 
and 
$\{g_{i}: (X, d)\to (X_{i}, d_{i})\}_{i\in \cind}$
such that for each 
$i\in \cind$ 
the pair 
$(f_{i}, g_{i})$ 
is an 
$\epsilon_{i}$-approximation. 
Take a finite non-constant sequence 
$\{z_{i}\}_{i=1}^{N}$ 
in 
$(X, d)$. 
For a sufficiently large 
$i\in \cind$, 
the sequence 
$\{g_{i}(z_{k})\}_{k=1}^{N}$ 
is non-constant. 
Since 
$(X_{i}, d_{i})$ 
is 
$\delta$-uniformly disconnected,  
we obtain 
\[
\delta d(g_{i}(z_1),g_{i}( z_N))\le 
\max_{1\le k\le N}d(g_{i}(z_k), g_{i}(z_{k+1})).
\]
By letting 
$i\to \infty$, 
we obtain 
$ \delta d(z_1, z_N)\le \max_{1\le i\le N}d(z_i, z_{i+1})$. 
This implies that  
$(X, d)\in \mathscr{S}(\delta)$. 
Since 
\[
\setud=\bigcup_{\delta\in \qq\cap (0, 1)}\mathscr{S}(\delta), 
\]
we conclude that 
$\setud$ is 
$F_{\sigma}$ 
in 
$\grsp$. 
\end{proof}

\begin{lem}\label{lem:setup}
The set 
$\setup$ 
is 
$F_{\sigma}$ 
and meager 
in 
$\grsp$. 
\end{lem}
\begin{proof}
Since 
$\grsp\setminus \setup$ 
is dense 
(see Lemma \ref{lem:denseuvw}), 
it suffices to show that 
$\setup$ 
is 
$F_{\sigma}$. 
For 
$c\in (0, 1)$ 
and 
$t\in (0, \infty)$, 
let 
$\mathscr{S}(c, t)$ 
be the 
set of all metric spaces 
$(X, d)$ 
such that 
for all
 $x\in X$ 
 for all 
 $r\in (0, t)$
there exists 
$y\in X$ 
with 
$cr\le d(x, y)\le r$. 
We  prove that 
$\mathscr{S}(c, t)$ 
is 
closed in 
$\grsp$. 
Take a convergent sequence
 $\{(X_{i}, d_{i})\}_{i\in \cind}$ 
 in 
 $\mathscr{S}(c, t)$, 
 and let 
 $(X, d)$ 
 be its limit compact metric space. 
 Then, 
 by Lemma \ref{lem:approx}, 
  there exist a positive sequence 
$\{\epsilon_{i}\}_{i\in \cind}$ 
converging  to 
$0$, 
and  sequences 
$\{f_{i}: (X_{i}, d_{i})\to (X, d)\}_{i\in \cind}$ 
and 
$\{g_{i}: (X, d)\to (X_{i}, d_{i})\}_{i\in \cind}$
such that for each 
$i\in \cind$ 
the pair 
$(f_{i}, g_{i})$ 
is an 
$\epsilon_{i}$-approximation. 
 Take  arbitrary 
 $z\in X$, 
and for each 
$n\in \cind$ 
put 
$x_{i}=g_{i}(z)$. 
Then 
$x_{i}\in X_{i}$ 
and
$d(f_{i}(x_{i}), z)\le \epsilon_{i}$. 
By 
$(X_{i}, d_{i})\in \mathscr{S}(c, t)$, 
there exists 
$y_{i}$ 
with 
$cr\le d_n(x_{i}, y_{i})\le r$. 
Combining these inequalities and 
$|d_n(x_{i}, y_{i})-d(f_{i}(x_{i}), f_{i}(y_{i}))|\le \epsilon_{i}$, 
we obtain
\begin{align}\label{eq:up}
cr-2\epsilon_{i}\le 
d(z, f_{i}(y))
\le r+2\epsilon_{i}. 
\end{align}
By extracting a subsequence if necessary, 
we may assume that the sequence 
$\{f_{i}(y_{i})\}_{i\in \cind}$ 
is convergent.  
Let 
$w$ 
be its limit point. 
By letting 
$n\to \infty$, 
 by (\ref{eq:up}), 
we obtain 
$cr\le d(z, w)\le r$. 
This implies that 
$(X, d)\in \mathscr{S}(c, t)$, 
and hence 
$\mathscr{S}(c, t)$ 
is closed in 
$\grsp$. 
Since 
\[
\setup=
\bigcup_{c\in \qq\cap (0, 1), t\in \qq_{>0}}\mathscr{S}(c, t), 
\]
we conclude that 
$\setup$ 
is 
$F_{\sigma}$. 
\end{proof}

We now prove 
Theorems \ref{thm:typicality}
and \ref{thm:topdisuvw}. 
\begin{proof}[Proof of Theorem \ref{thm:typicality}]
By combining 
Lemmas \ref{lem:setd}, \ref{lem:setud}, 
and 
\ref{lem:setup}, 
we obtain Theorem \ref{thm:typicality}. 
\end{proof}

\begin{proof}[Proof of Theorem  \ref{thm:topdisuvw}]
From Theorem \ref{thm:typicality}  
and Lemma \ref{lem:denseuvw}, 
and  the fact that 
the intersection of 
an 
$F_{\sigma}$ 
set and a 
$G_{\delta}$ 
set of a metric space is 
$F_{\sigma\delta}$ 
and 
$G_{\delta\sigma}$, 
Theorem \ref{thm:topdisuvw} follows. 
\end{proof}


\section{Constructions of metric spaces}\label{sec:const}

In this section, 
we prepare constructions of metrics spaces 
to 
prove Theorem \ref{thm:q-emb}.

Let 
$\mathcal{X}=\{(X_{i},d_{i})\}_{i\in \cind}$ 
be 
a  family of metric spaces
satisfying the inequality
$\delta_{d_{i}}(X_{i})\le 2^{-i-1}$ for all $i\in \cind$. 
We put
\[
T(\mathcal{X})=\{\infty\}\sqcup \coprod_{i\in \cind}X_i, 
\]
and define a symmetric function 
$d_{\mathcal{X}}: (T(\mathcal{X}))^{2}\to [0, \infty)$ by 
\[
d_{\mathcal{X}}(x,y)=
	\begin{cases}
		d_i(x,y) & \text{if $x,y\in X_i$ for some $i$,}\\
		|2^{-i} -2^{-j}| & \text{if $x\in X_i,y\in X_j$ for some $i\neq j$, }\\
		2^{-i} & \text{if $x=\infty, y\in X_i$ for some $i$.}
	\end{cases}
\]
Then 
$d_{\mathcal{X}}$ 
is a metric on 
$T(\mathcal{X})$. 
This construction is 
 a specific version of  
 the telescope space defined in 
\cite{Ishiki2019}. 
The space 
$(T(\mathcal{X}), d_{\mathcal{X}})$ 
is the 
same as
 the telescope space 
$\left(T(\mathcal{X}, \mathcal{R}), d_{(\mathcal{X}, \mathcal{R})}\right)$, 
where 
$\mathcal{R}$ 
is the telescope base defined in 
\cite[Definition 3.2]{Ishiki2019}. 
The symbol 
$\mathcal{R}$ 
is a pair of the space 
$\{0\}\cup\{\, 2^{-n}\mid n\in\cind\, \}$ 
with the Euclidean metric  
and its numbering map. 
We review the basic properties of this construction. 

\begin{prop}\label{prop:teleprop}
Let $\mathcal{X}=\{(X_{i}, d_{i})\}_{i\in \cind}$ 
be a sequence of metric spaces. 
with 
$\delta_{d_{i}}(X_{i})\le 2^{-i-1}$. 
Then, 
the following hold true. 
\begin{enumerate}
\item 
If there exists 
$N\in \cind$ 
such that each 
$(X_{i}, d_{i})$ 
is 
$N$-doubling, 
then 
$(T(\mathcal{X}), d_{\mathcal{X}})$
is doubling. 
\item 
If there exists  
$\delta\in (0, 1)$ 
such that 
each 
$(X_{i}, d_{i})$ 
is 
$\delta$-uniformly disconnected, 
then 
$(T(\mathcal{X}), d_{\mathcal{X}})$
is uniformly disconnected. 
\item 
If there exist 
$c\in (0, 1)$ 
and 
$M\in (0, \infty)$ 
such that each 
$(X_{i}, d_{i})$ 
is 
$c$-uniformly perfect and 
satisfies 
$M\cdot 2^{-i}\le \delta_{d_{i}}(X_{i})$, 
then 
$(T(\mathcal{X}), d_{\mathcal{X}})$
is uniformly perfect. 
\end{enumerate}
\end{prop}
\begin{proof}
The statements (1), (2), and (3)
follow from 
 Proposition 3.8, 
 Proposition 3.9, 
 Proposition 3.10 in 
 \cite{Ishiki2019}, respectively. 
\end{proof}

We denote by 
$2^{\omega}$ 
the set of all maps from 
$\zz_{\ge0}$ 
into 
$\{0, 1\}$. 
We define  
 $v: 2^{\omega}\times 2^{\omega}\to [0, \infty]$ 
by 
$v(x, y)=\min\{\, n\in \zz_{\ge 0}\mid x(n)\neq y(n)\, \}$ if 
$x\neq y$; 
otherwise 
$v(x, y)=\infty$. 
For each 
$c\in (0, 1)$, 
we define 
$g_{c}:\zz_{\ge 0}\cup\{\infty\}\to [0, \infty)$
by 
$g_{c}(n)=c^{n}$ 
if 
$n\in \zz_{\ge 0}$ 
and 
$g_{c}(\infty)=0$.  
We define a metric 
$\beta_{c}$ 
on 
$2^{\omega}$ 
by 
$\beta_{c}(x, y)=g_{c}(v(x, y))$.
For every 
$c\in (0, 1)$, 
the metric space 
$(2^{\omega}, \beta_{c})$ 
is a Cantor space. 
This metric space is a specific version of 
a \emph{sequentially metrized 
 Cantor space} defined  in the author's paper  \cite{Ishiki2019}. 
 The following is  implicitly shown  in 
 the proof of \cite[Lemma 6.3]{Ishiki2019}. 
 \begin{lem}\label{lem:smdoubling}
 Let 
 $c\in (0, 1)$. 
 If 
 $N\in (0, \infty)$ 
 satisfies that for all 
 $k\in \cind$ 
 we have 
 $\card(\{\, n\in \cind\mid c^{k}/2\le c^{n}\le c^{k}\, \})
 \le N$,
 then 
 $(2^{\omega}, \beta_{c})$ 
 is 
 $(2^{N+1})$-doubling. 
 In particular, 
$(2^{\omega}, \beta_{c})$
 is 
 $(2^{\log2/\log c^{-1}+2})$-doubling. 
 \end{lem}

\begin{lem}\label{lem:smprop}
Let 
$c\in (0, 1)$. 
Then 
\begin{enumerate}
\item 
The space 
$(2^{\omega}, \beta_{c})$ 
is 
$(2^{\log2/\log c^{-1}+2})$-doubling.
\item 
The space 
$(2^{\omega}, \beta_{c})$
 is 
an ultrametric space.
\item 
The space 
$(2^{\omega}, \beta_{c})$ 
is 
$c$-uniformly perfect.
\item 
For all 
$x\in 2^{\omega}$, 
and 
for all 
$r\in (0, \infty)$, 
we have\label{item:daidai}
 \[
 \dim_{A}(B(x, r), \beta_{c})=\log 2/\log c^{-1}. 
\]

\end{enumerate}
\end{lem}
\begin{proof}
Lemma \ref{lem:smdoubling} implies 
the statement (1). 
The statements (2) and (3) follow from 
the definition of 
$\beta_{c}$. 
The statement (4) follows from 
the fact that 
$(B(x, r), \beta_{c})$ 
is isometric to 
$(2^{\omega}, c^{m} \beta_{c})$ 
for some 
$m\in \cind$. 
\end{proof}

\begin{df}\label{df:tail}
Let 
$\mathcal{P}=
\{(P_{i, j}, p_{i, j})\}_{i\in \{1, 2, 3\}, j\in \cind}$
be a sequence of compact metric spaces  indexed by 
$\{1, 2, 3\}\times \cind$. 
For $q\in \qcube$, 
we put 
$\mathcal{P}_{q}=(\mathcal{P}, q)$. 
We now define a metric space 
$(U(\mathcal{P}_{q}), e_{\mathcal{P}_{q}})$ 
induced from 
$\mathcal{P}_{q}$ 
by using the telescope construction. 
We define 
$l: \rr\to [0, \infty)$ 
by 
$l(x)=\sqrt{2-2\cos(x)}$, 
and 
define a homeomorphism 
$\theta: [0, 1]\to [\pi/6, \pi/3]$ 
by 
$\theta(t)=(\pi/6)(t+1)$. 
Assume that 
$\delta_{p_{i, j}}(P_{i, j})\le 2^{-j-1}l(\pi/6)$ for all $i\in \{1, 2, 3\}$ and $j\in \cind$. 
For 
$q=\{q_{i}\}_{i\in \cind}\in \qcube$ 
and  
$j\in \cind$, 
define 
a metric 
$e_{j, q}\in \met\left(\widehat{3}\right)$ 
by 
\[
e_{j, q}(a, b)=
\begin{cases}
2^{-j-1} & \text{if  
$\{a, b\}=\{1, 2\}$ or $ \{2, 3\}$;}\\
2^{-j-1}\cdot l(\theta(q_{j})) & 
\text{if $\{a, b\}=\{1, 3\}$. }
\end{cases}
\]
The metric space 
$\left(\widehat{3}, e_{j, q}\right)$
is the set of vertices of the isosceles triangle whose
apex angle is 
$\theta(q_{j})$ 
and 
whose length of the legs is 
$2^{-j-1}$. 
For each 
$j\in \cind$, 
we put 
$T_{j}=P_{1, j}\sqcup P_{2, j}\sqcup P_{3, j}$
and 
let 
$k_{j}$ 
be a metric stated in 
Lemma \ref{lem:amaldiam} 
induced from 
$\{(P_{i, j}, p_{i, j})\}_{i\in \{1, 2, 3\}}$ 
and 
$e_{j, q}\in \met\left(\widehat{3}\right)$. 
Put 
$\widetilde{\mathcal{P}_{q}}=\{(T_{j}, k_{j})\}_{j\in \cind}$. 
Then we have 
$\delta_{k_{j}}(T_{j})= 2^{-j-1}$. 
We define 
\[
(U(\mathcal{P}_{q}), e_{\mathcal{P}_{q}})
=\left(T\left(\widetilde{\mathcal{P}_{q}}\right), 
d_{\widetilde{\mathcal{P}_{q}}}\right). 
\]
\end{df}

\begin{figure}[h]
\begin{minipage}{0.45\columnwidth}
\centering 
\includegraphics[width=\columnwidth]{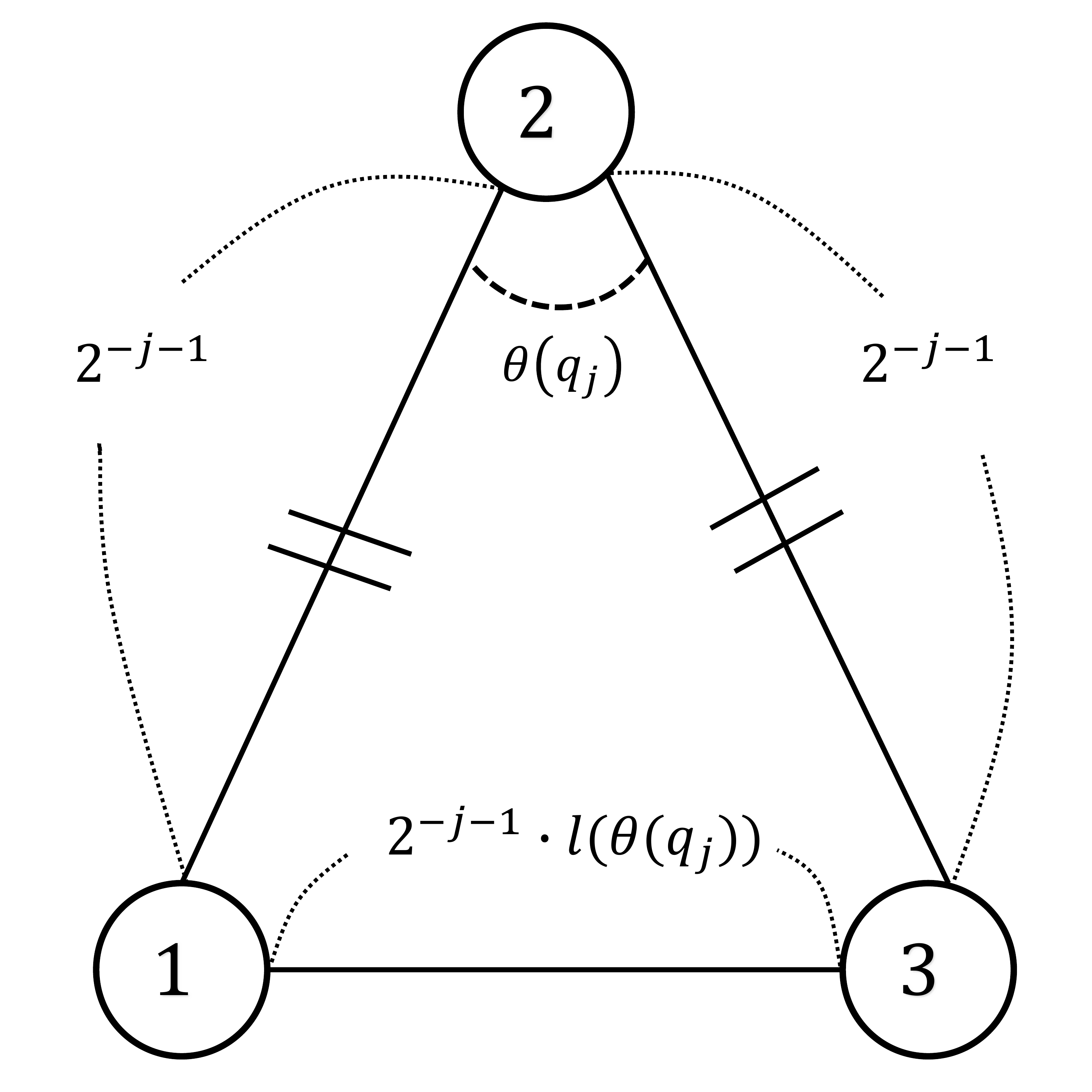}
\caption{$\left(\widehat{3}, e_{j, q}\right)$}
\end{minipage}
\begin{minipage}{0.45\columnwidth}
\centering 
\includegraphics[width=\columnwidth]{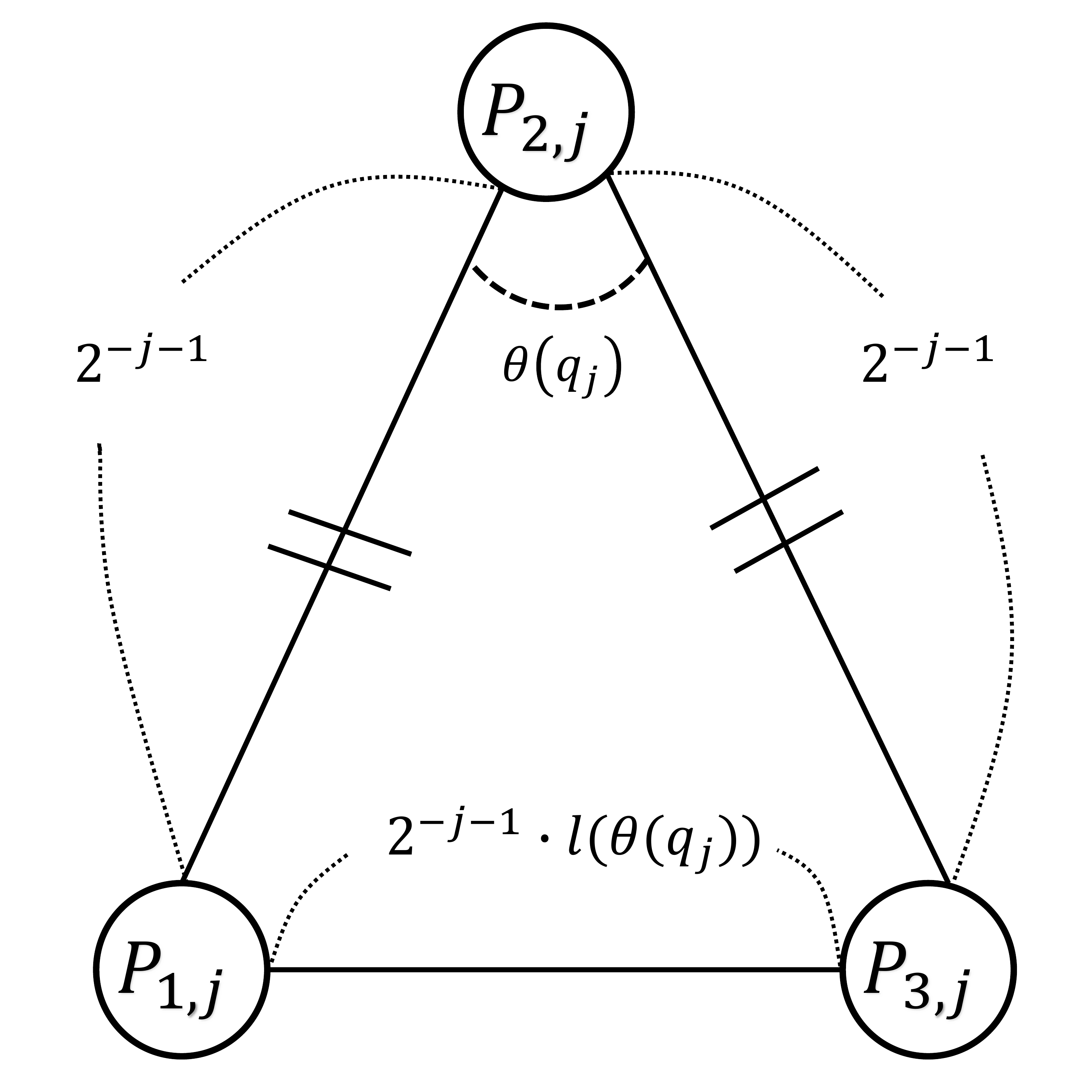}
\caption{$(T_{j}, k_{j})$
}
\end{minipage}
\end{figure}

\begin{figure}[h]
\centering 
\includegraphics[width=\columnwidth]{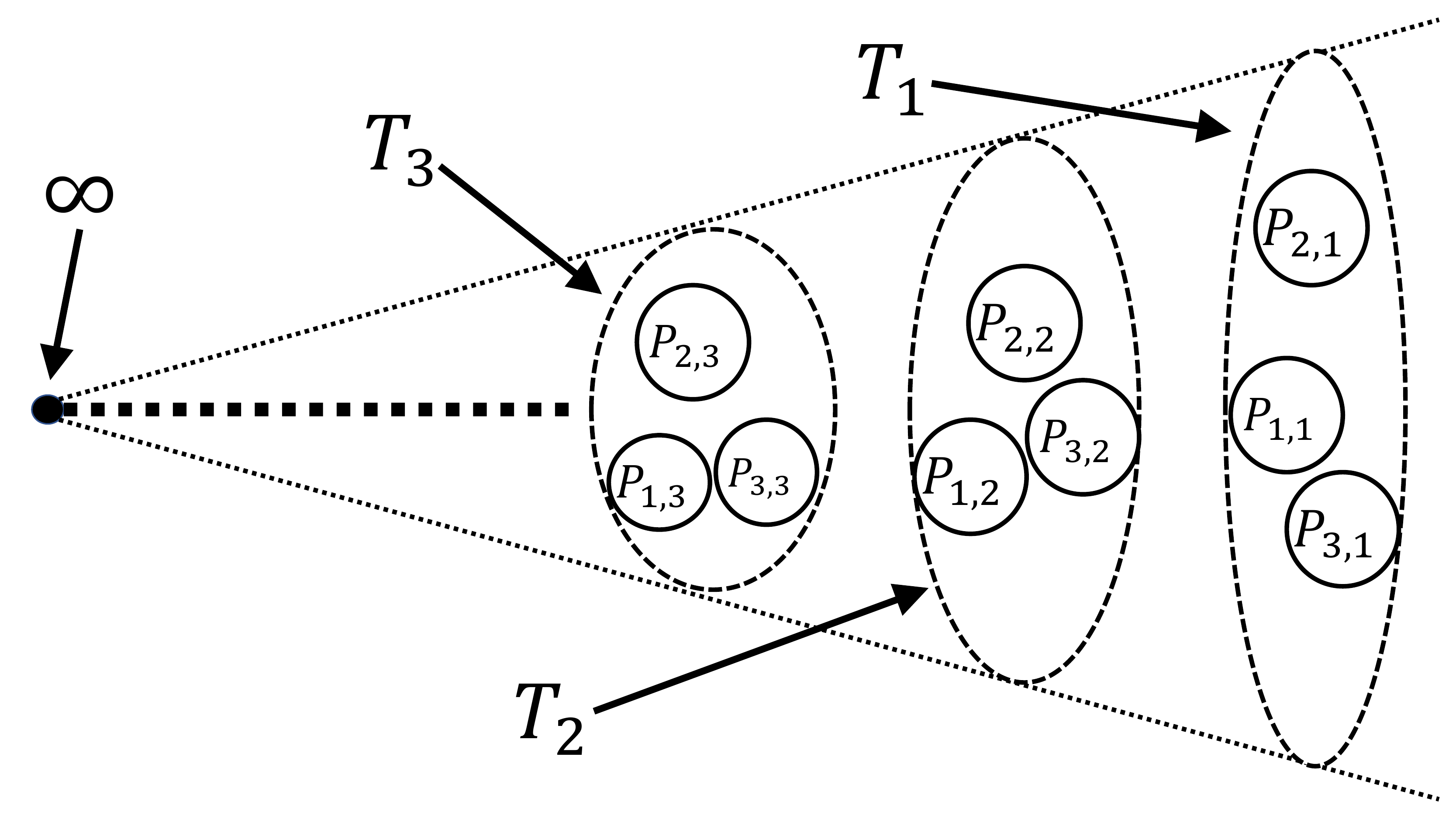}
\caption{$(U(\mathcal{P}_{q}), e_{\mathcal{P}_{q}})$ in Definition \ref{df:tail}
}
\end{figure}

By Proposition \ref{prop:teleprop}, 
we obtain:
\begin{prop}\label{prop:uprop}
Let $\mathcal{P}=\{(P_{i, j}, p_{i, j})\}_{i\in \{1, 2, 3\},  j\in \cind}$ be a sequence of metric spaces with 
$\delta_{p_{i, j}}(P_{i, j})\le 2^{-j-1}l(\pi/6)$. 
Let 
$q\in \qcube$. 
Then
we have:
\begin{enumerate}
\item 
If there exists 
$N\in \cind$ 
such that each 
$(P_{i, j}, p_{i, j})$ 
is 
$N$-doubling, 
then 
$(U(\mathcal{P}_{q}), e_{\mathcal{P}_{q}})$
is doubling. 
\item 
If there exists  
$\delta\in (0, 1)$ 
such that 
each 
$(P_{i, j}, p_{i, j})$  
is 
$\delta$-uniformly disconnected, 
then 
$(U(\mathcal{P}_{q}), e_{\mathcal{P}_{q}})$ 
is 
uniformly disconnected. 
\item 
If there exists 
$c\in (0, 1)$ 
such that 
each 
$(P_{i, j}, p_{i, j})$ 
is 
$c$-uniformly perfect, 
then 
$(U(\mathcal{P}_{q}), e_{\mathcal{P}_{q}})$ 
is uniformly perfect. 
\end{enumerate}
\end{prop}

The following proposition  states the continuity of 
$\{(U(\mathcal{P}_{q}), e_{\mathcal{P}_{q}})\}_{q\in \qcube}$.

\begin{prop}\label{prop:Uconti}
Let 
$\mathcal{P}=
\{(P_{i, j}, p_{i, j})\}_{i\in \{1, 2, 3\}, j\in \cind}$
be a sequence with 
$\delta_{p_{i, j}}(P_{i, j})\le 2^{-j-1}l(\pi/6)$.
Then the map $G:\qcube \to \grsp$ 
defined by 
$G(q)=(U(\mathcal{P}_{q}), e_{\mathcal{P}_{q}})$ 
is 
continuous. 
\end{prop}
\begin{proof}
For 
$q\in \qcube$, 
put 
$(X_{q}, d_{q})=(U(\mathcal{P}_{q}), e_{\mathcal{P}_{q}})$.
Note that 
$X_{q}=X_{r}$ 
for all 
$q, r\in \qcube$. 
We fix 
$q\in \qcube$, 
and take arbitrary 
$\epsilon\in (0, \infty)$. 
Since 
$l\circ \theta: [0, 1]\to \rr$ 
is 
uniformly continuous, 
there exists 
$\delta\in (0, \infty)$ 
such that 
for all 
$s, t\in [0, 1]$ 
with 
$|s-t|\le \delta$, 
we obtain 
$|l\circ \theta(s)-l\circ \theta(t)|\le \epsilon$. 
Take 
$n\in \cind$ 
such that 
$2^{-n}\le \min\{\delta, \epsilon\}$. 
We denote by 
$V$ 
the set of all 
$r\in \qcube$ 
satisfying that 
for every 
$i\in \{1, 2, \dots, n\}$ 
we obtain 
$|q_{i}-r_{i}|\le 2^{-n}$. 
Note that 
$V$ 
is a neighborhood of
 $q$
  in 
  $\qcube$. 
We then estimate 
$\dis(\Delta_{X_{q}})$. 
By the definition of the metric space 
$(U(\mathcal{P}_{q}), e_{\mathcal{P}_{q}})$, 
for every 
$r\in V$, 
the quantity 
$|d_{q}(x, y)-d_{r}(x, y)|$ 
only takes values in 
the set 
$\{|l(\theta(q_{i}))-l(\theta(r_{i}))|\cdot 2^{-i-1}\mid i\in \cind\}$. 
If 
$i\in \{1, \dots, n\}$, by $|q_{i}-r_{i}|\le \delta$, 
we obtain 
$|l(\theta(q_{i}))-l(\theta(r_{i}))|\cdot 2^{-i-1}\le \epsilon$.
If 
$i> n$,   
using  
$l\circ \theta(x)\le 1$ 
for all 
$x\in [0, 1]$, 
we obtain  
$|l(\theta(q_{i}))-l(\theta(r_{i}))|\cdot 2^{-i-1}\le 
2\cdot 2^{-n}
\le
2 \epsilon$. 
This implies that 
$\dis(\Delta_{X_{q}})\le 2\epsilon$. 
Thus, we conclude that  for every 
$r\in V$, 
we obtain 
$\grdis(G(q), G(r))\le \epsilon$, 
and hence
$G$ 
is continuous. 
\end{proof}

A topological space is said to be \emph{perfect} if 
it has no isolated points. 
For a metric space 
$(X, d)$, 
and 
$v\in [0, \infty)$ 
we denote by 
$\aaa(v, X, d)$ 
the set of all 
$x\in X$ 
for which 
there exists 
$r\in (0, \infty)$ 
such that 
for every 
$\epsilon \in (0, r)$ 
we have 
$\dim_{A}(B(x, \epsilon), d)=v$. 
Note that if $(X, d)$ and $(Y, e)$ are isometric to each other, then so are $\aaa(v, X, d)$ and $\aaa(v, Y, e)$
for all $v\in [0, \infty)$. 

Under certain assumptions on  
a family 
$\mathcal{P}$, 
we can prove that 
the family 
$\{(U(\mathcal{P}_{q}), e_{\mathcal{P}_{q}})\}_{q\in \qcube}$ 
are not isometric to each other. 
\begin{prop}\label{prop:Uiso}
Let 
$\mathcal{P}=
\{(P_{i, j}, p_{i, j})\}_{i\in \{1, 2, 3\}, j\in \cind}$
be a sequence of compact metric spaces
 with 
$\delta_{p_{i, j}}(P_{i, j})\le 2^{-j-1}l(\pi/6)$. 
We assume that 
either of the following conditions holds true:
\begin{enumerate}

\item Each  
$(P_{i, j}, p_{i, j})$ 
is the  one-point metric space. \label{item:first}

\item 
For all $i\in \{1, 2, 3\}$ 
and 
$j\in \cind$, 
we have:
\begin{enumerate}
\item the space 
$(P_{i, j}, p_{i, j})$ 
is perfect;\label{item:aaa}
\item 
$\aaa(
\dim_{A}(P_{i, j}, p_{i, j}), 
P_{i, j}, p_{i, j}
)
=
(P_{i, j}, p_{i, j})$;\label{item:bbb}
\item the values
 $\{\dim_{A}(P_{i, j}, p_{i, j})\}_{i, \in\{1, 2, 3,\}, j\in \cind }$
 are different from  each other. \label{item:ccc}
\end{enumerate}
\label{item:second}

\end{enumerate}
Then for all  
$q, r\in \qcube$ 
with 
$q\neq r$, 
and 
for all 
$K, L\in (0, \infty)$, 
the spaces 
$(U(\mathcal{P}_{q}), K\cdot e_{\mathcal{P}_{q}})$ 
and 
$(U(\mathcal{P}_{r}), L\cdot e_{\mathcal{P}_{r}})$ 
are not isometric to each other. 
\end{prop}
\begin{proof}
Put 
$(X, d)=(U(\mathcal{P}_{q}), K\cdot e_{\mathcal{P}_{q}})$
and 
$(Y, e)=(U(\mathcal{P}_{r}), L\cdot e_{\mathcal{P}_{r}})$. 
For the sake of contradiction, suppose that there exists an isometry 
$I: (X, d)\to (Y, e)$.

We
first  assume that 
the condition (\ref{item:first}) is true. 
By the definition, 
the spaces
$(X, d)$ 
and 
$(Y, e)$
have unique accumulation points, 
say 
$\omega$, 
$\omega'$, 
respectively. 
Then 
$I(\omega)=\omega'$. 
Thus, 
we obtain 
$\{\,d(\omega, x)\mid 
x\in X
\, \}
=
\{\,d(\omega', y)\mid 
y\in Y
\, \}
$. 
Since 
$\{\,d(\omega, x)\mid 
x\in X
\, \}
=
\{0\}\cup \{\, K\cdot 2^{-i}\mid i\in \cind\, \}$, and 
 $\{\,e(\omega', y)\mid 
y\in Y
\, \}
=
\{0\}\cup \{\, L\cdot 2^{-i}\mid i\in \cind\, \}$, 
we obtain 
$K=L$. 
For each 
$i\in \cind$, 
the set 
$\{x\in  X 
\mid d(\omega, x)=K2^{-i}\}$
consists of three points,
 and they form an isosceles triangle. 
The apex angle of this isosceles triangle is equal to 
$\theta(q_{i})$, 
and 
$I(\{x\in  X 
\mid d(\omega, x)=K2^{-i}\})=
\{y\in  Y
\mid e(\omega', y)=L2^{-i}\}$. 
Then 
we obtain 
$\theta(q_{i})=\theta(r_{i})$, 
and hence 
$q=r$. 
This is a contradiction.

Second, we   assume that the condition 
(\ref{item:second}) holds true. 
Note that 
both 
$(X, d)$ 
and 
$(Y, e)$ 
contain $P_{i, j}$ for all 
$(i, j)\in\{1, 2, 3 \}\times \cind$ and 
the element
represented as 
$\infty$. 
For each $
(i, j)\in \{1, 2, 3\}\times \cind$, 
we put
 $\hensu_{1}(i, j)=\aaa(
\dim_{A}(P_{i, j}, p_{i, j}), 
X, d
)$
and 
$\hensu_{2}(i, j)=\aaa(
\dim_{A}(P_{i, j}, p_{i, j}), 
Y, e
)$. 
By the assumption (\ref{item:bbb}) and (\ref{item:ccc}), 
and
 by  the 
fact that each $P_{i, j}$ is open in 
$X$ and $Y$, respectively, 
for all $k\in \{1, 2\}$ we have 
\[
\hensu_{k}(i, j)=
\begin{cases}
P_{i, j}& \text{if $\infty\not\in \hensu_{k}(i, j)$}\\
\{\infty\}\cup P_{i, j}& \text{if $\infty\in \hensu_{k}(i, j)$. }
\end{cases}
\]
By  Definition \ref{df:tail}, 
 for all $k\in \{1, 2\}$,  
if $\infty\in \hensu_{k}(i, j)$, 
then $\infty$ is an isolated point of $\hensu_{k}(i, j)$. 
 Then, 
 since each $P_{i, j}$ is perfect (the assumption (\ref{item:aaa})), 
for each 
$(i, j)\in \{1, 2, 3\}\times \cind$, 
we can take  a non-isolated  point 
$a_{i, j}\in \hensu_{1}(
i, j)$. 
Note that 
$a_{i, j}\neq \infty$ and
$a_{i, j}\in P_{i, j}$. 
Since $I$ is an isometry and 
$\aaa$ is invariant under isometries, 
the point $I(a_{i, j})\in Y$ is a non-isolated point 
and  $I(a_{i, j})\in \hensu_{2}(i, j)$. 
Since $\infty\not\in \hensu_{2}(i, j)$ or 
$\infty$ is an isolated point of $\hensu_{2}(i, j)$, 
we have 
$I(a_{i, j})\in P_{i, j}\subset Y$.  
Put 
$\mathcal{Q}=\{(\{a_{i, j}\}, d_{i, j})\}_{i\in \{1, 2,3\}, j\in \cind}$, 
where 
$d_{i, j}$ 
is the trivial metric.
Then 
$(U(\mathcal{Q}_{q}), e_{\mathcal{Q}_{q}})$ 
and 
 $(U(\mathcal{Q}_{r}), e_{\mathcal{Q}_{r}})$ 
 are isometric to each other. 
Thus by the first case, we obtain
$q=r$. 
This leads to  the proposition. 
\end{proof}

\section{Geodesics}\label{sec:geo}
In this section, 
we prove Theorem \ref{thm:q-emb}
 and 
\ref{thm:q-embinfinite}. 
As its applications, we next prove 
Theorem \ref{thm:branch}, 
\ref{thm:topemb}, and
\ref{thm:locinfinite}.

The following construction of geodesics  using  optimal closed correspondences is presented  in 
\cite{CM2018} and \cite{IIT2016}.  
\begin{prop}\label{prop:canonicalgeo}
Let 
$(X, d)$ 
and 
$(Y, e)$ 
be compact metric spaces with 
$\grdis((X, d), (Y, e))>0$. 
Let 
$R\in \relacc_{opt}(X, d,  Y, e)$. 
For each 
$s\in (0, 1)$, 
define a metric 
$D_{s}$ 
on 
$R$ 
by  
$D_{s}((x, y), (u, v))=(1-s)d(x, u)+se(y, v)$. 
We define 
$\gamma:[0, 1]\to \grsp$ 
by 
\[
\gamma(s)=
\begin{cases}
(X, d) & \text{if $s=0$;}\\
(Y, e) & \text{if $s=1$;}\\
(R, D_{s}) & \text{if $s\in (0, 1)$. }
\end{cases}
\]
Then 
$\gamma$ 
is a geodesic in 
$(\grsp, \grdis)$ 
from 
$(X, d)$ 
to  
$(Y, e)$.
Moreover, 
\begin{enumerate}
\item 
For every 
$s\in (0, 1)$,
the set 
$\{\, (x, z)\in X\times R\mid \pi_{X}(z)=x\, \}$ 
is in 
$\relacc_{opt}(\gamma(0), \gamma(s))$. 
\item 
For every 
$s\in (0, 1)$,
the set 
$\{\, (z, x)\in R\times Y\mid \pi_{Y}(z)=y\, \}$ 
is in 
$\relacc_{opt}(\gamma(s), \gamma(1))$. 
\item 
For all 
$s, t\in (0, 1)$, 
we have 
$\Delta_{R}\in 
\relacc_{opt}(\gamma(s), \gamma(t))$. 
\end{enumerate}

\end{prop}
The following is a useful criterion of  geodesics. 
The  proof is presented  in 
\cite[Lemma 1.3]{CM2018}. 

\begin{lem}\label{lem:crigeo}
Let 
$\gamma: [0, 1]\to X$ 
be a curve. 
If for every 
$s, t\in [0, 1]$, 
we have 
$d(\gamma(s), \gamma(t))
\le |s-t|\cdot d(\gamma(0), \gamma(1))$, 
then $\gamma$ 
is a geodesic. 
\end{lem}

In Propositions \ref{prop:prodgeo} 
and \ref{prop:tailedgeo}, 
we construct  modifications of the  geodesic stated  in Proposition 
\ref{prop:canonicalgeo}. 
In the proof of Theorem \ref{thm:q-emb}, 
Proposition  \ref{prop:prodgeo}  is used to obtain 
a geodesic consisting  of perfect metric spaces. 
Proposition \ref{prop:tailedgeo}
 is used to attach $\{(U_{q}, e_{q})\}_{q\in \qcube}$ in Definition  \ref{df:tail} to a geodesic constructed in 
 Proposition  \ref{prop:prodgeo}. 
\begin{prop}\label{prop:prodgeo}
Let 
$(X_{0}, d_{0})$
 and 
 $(X_{1}, d_{1})$ 
 be compact metric spaces 
 with  
 $\grdis((X_{0}, d_{0}), (X_{1}, d_{1}))>0$. 
 Let 
 $f: [0, 1]\to \grsp$
 be a geodesic from 
 $(X_{0}, d_{0})$ 
 to 
 $(X_{1}, d_{1})$. 
 Let 
 $(Y, e)$ 
 be in 
 $\grsp$, 
 and 
 $L\in (0, \infty)$ satisfy 
 $L\delta_{e}(Y)\le 2\grdis((X_{0}, d_{0}), (X_{1}, d_{1}))$. 
 Let 
 $A$ 
 be a closed subset of 
 $[0, 1]$ 
 with 
 $\{0, 1\}\subset A$, 
 and 
 let 
 $\zeta: [0, 1]\to [0, \infty)$ 
 be an 
 $L$-Lipschitz map 
 with 
 $\zeta^{-1}(0)=A$. 
Put 
$f(s)=(X_{s}, d_{s})$. 
 We define a map 
 $F: [0, 1]\to \grsp$ 
 by 
 \[
 F(s)=
 \begin{cases}
f(s) & \text{if $s\in A$;}\\
 (X_{s}\times Y, d_{s}\times_{\infty}(\zeta(s)\cdot e))
 & 
 \text{if $s \in [0, 1]\setminus A$. } 
 \end{cases}
 \]
 Then 
 $F$ 
 is a geodesic from 
 $(X_{0}, d_{0})$ 
 to 
 $(X_{1}, d_{1})$. 
Put 
$F(s)=(Z_{s}, D_{s})$. 
 Moreover, 
 \begin{enumerate}
 \item 
 for all  
 $s, t\in [0, 1]\setminus A$, 
 and  for all 
 $R\in \relacc_{opt}(f(s), f(t))$, 
 the set 
 \[
\{\, ((x, a), (x', a))\in Z_{s}\times Z_{t}\mid(x, x')\in R, a\in Y\, \}
\]
is in 
$\relacc_{opt}(F(s), F(t))$;
\item for all 
$s\in A$ 
and 
$t\in[0, 1]\setminus  A$, 
and 
for all 
$R\in \relacc_{opt}(f(s), f(t))$, 
 the set 
 \[
\{\, (x, (x', a))\in Z_{s}\times Z_{t}\mid(x, x')\in R, a\in Y\, \}
\]
is in 
$\relacc_{opt}(F(s), F(t))$.

 \end{enumerate}
 
\end{prop}
\begin{proof}
By Lemma \ref{lem:crigeo}, 
it suffices to show that 
\begin{align}\label{eq:ghgeodesic}
\gh(F(s), F(t))\le |s-t|\cdot \gh((X_{0}, d_{0}), (X_{1}, d_{1})).
\end{align}
We assume that $s, t\in [0, 1]\setminus A$, 
and 
take an optimal correspondence   
$R\in \relacc_{opt}(X_{s}, d_{s}, X_{t}, d_{t})$. 
Put
\[
U=
\{\, ((x, a), (x', a))\in Z_{s}\times Z_{t}\mid(x, x')\in R, a\in Y\, \}
\]
Then, 
$U\in \rela(Z_{s}, Z_{t})$. 
For all 
$((x, a), (x', a)), ((y, b), (y', b))\in U$, 
by Lemma \ref{lem:max} 
we obtain 
\begin{align*}
&|D_{s}((x, a), (y, b))-D_{t}((x', a), (y', b))|
\\
&\le 
|d_{s}(x, y)-d_{t}(x', y')|\lor 
|(\zeta(s)e)(a, b)-(\zeta(t)e)(a, b)|. 
\end{align*}
Since 
$R$ 
is optimal, 
we obtain 
\[
|d_{s}(x, y)-d_{t}(x', y')|\le 2|s-t|\gh(X_{0}, d_{0}, X_{1}, d_{1}). 
\]
By the assumption, 
\begin{align*}
 |(\zeta(s)e)(a, b)-(\zeta(t)e)(a, b)|
& =|\zeta(s)-\zeta(t)|e(a, b)
 \le |s-t| L \delta_{e}(Y)\\
 &\le 
2|s-t|\grdis((X_{0}, d_{0}), (X_{1}, d_{1})). 
\end{align*}
Thus, 
$\dis(U)
\le 2|s-t|\gh((X_{0}, d_{0}), (X_{1}, d_{1}))$. 
This implies 
the inequality (\ref{eq:ghgeodesic}). 
The remaining  cases can be solved  in  a similar way. 
\end{proof}

\begin{prop}\label{prop:tailedgeo}
Let 
$(X_{0}, d_{0})$
 and 
 $(X_{1}, d_{1})$  
 be compact metric spaces 
 with 
 $\grdis((X_{0}, d_{0}), (X_{1}, d_{1}))>0$. 
 Let 
 $f: [0, 1]\to \grsp$
 be a geodesic from 
 $(X_{0}, d_{0})$ 
 to 
 $(X_{1}, d_{1})$. 
 Let 
 $(Y, e)\in \grsp$,  
 and let 
 $L\in (0, \infty)$ 
 satisfy 
 $L\delta_{e}(Y)\le 2\grdis((X_{0}, d_{0}), (X_{1}, d_{1}))$. 
 Let 
 $A$ 
 be a closed subset of 
 $[0, 1]$ 
 with 
 $\{0, 1\}\subset A$, 
 and 
 let 
 $\zeta: [0, 1]\to [0, \infty)$ 
 be an 
 $L$-Lipschitz map 
 with 
 $\zeta^{-1}(0)=A$. 
 We assume that there exists a set 
 $R$ 
 such that 
for all 
$s\in (0, 1)$ 
we can represent 
$f(s)=(R, d_{s})$. 
Let 
$o\in R$, 
and 
$\omega\in Y$. 
We define 
\[
Z=R\times \{\omega\}\cup\{o\}\times Y, 
\]
and define a map 
$F: [0, 1]\to \grsp$ 
by 
 \[
 F(s)=
 \begin{cases}
  f(s) & \text{if $s\in A$;}\\
 (Z, d_{s}\times_{\infty}(\zeta(s)\cdot e)) & \text{if $s\in [0, 1]\setminus A$.}
 \end{cases}
 \]
 If for all 
 $s, t\in (0, 1)$, 
 there exists  
 $K\in \relacc_{opt}(R, d_{s}, R, d_{t})$ 
 with 
 $(o, o)\in K$, 
then 
$F$ 
is a geodesic from 
$(X_{0}, d_{0})$ 
to 
$(X_{1}, d_{1})$.

\end{prop}
\begin{proof}
Since 
$F$ 
is continuous, 
by Lemma \ref{lem:crigeo}, 
it suffices to show the equality 
$\gh(F(s), F(t))\le |s-t|\cdot \gh(X_{0},d_{0},  X_{1}, d_{1})$
 for all 
 $s, t\in (0, 1)$. 
Take  
$K\in \relacc_{opt}(R, d_{s}, R, d_{t})$ 
with 
$(o, o)\in K$. 
We define a correspondence 
$U\in \rela(Z, Z)$ 
by 
\begin{align*}
U=
\{\, ((x,\omega), (x', \omega))\mid(x, x')\in K\, \}
\cup\{\,((o, y), (o, y))\mid y\in Y \, \}. 
\end{align*}
The rest of the proof can be  presented 
in the same way as 
the  proof of Proposition \ref{prop:prodgeo}. 
Remark that to obtain the inequality 
 \[
|d_{s}(x, y)-d_{t}(x', y')|\le 2|s-t|\gh((X_{0}, d_{0}), ( X_{1}, d_{1})),
\]
we need to use the assumption that  
$(o, o)\in K$. 
\end{proof}


\begin{proof}[Proof of Theorem \ref{thm:q-emb}]
Let 
$A$ 
be a closed subset of 
$[0, 1]$. 
Let 
$(u, v, w)\in \{0, 1, 2\}^{3}$. 
Let 
$(X, d), (Y, e)\in \mathscr{X}(u, v, w)$ 
with 
$\grdis((X, d), (Y, e))>0$. 
We  only need to  consider 
 the case of 
$(u, v, w)\in \{0, 1\}^{3}$. 

Let 
$R\in \relacc_{opt}(X, d, Y, e)$, 
and for each 
$s\in (0, 1)$, 
put  
\[
d_{s}=(1-s)\cdot d+s\cdot e.
\]  
Then by  Proposition \ref{prop:canonicalgeo}, 
the map 
$\gamma:[0, 1]\to \grsp$ 
defined by 
\[
\gamma(s)=
\begin{cases}
(X, d) & \text{if $s=0$;}\\
(Y, e) & \text{if $s=1$;}\\
(R, d_{s}) & \text{if $s\in (0, 1)$}
\end{cases}
\]
is a geodesic 
from 
$(X, d)$ 
to 
$(Y, e)$.

To construct a branching bunch of geodesics, 
in the following, 
we first construct an ``expansion'' $\{(W, E_{s})\}_{s\in (0, 1)}$ of 
the geodesics $\gamma$ such that each $(W, E_{s})$ is the  product space  of 
$\gamma(s)$ and a Cantor  metric space $(C, h)$ in  $\mathscr{X}(u, v, w)$.  Such a  choice of $(C, h)$ is needed to prove $F(s, q)\in \mathscr{X}(u, v, w)$. 
By the construction, each $(W, E_{s})$ is perfect and this property is 
convenient for our aim (note that each $\gamma(s)$ is 
not always perfect).
If 
$(u, v, w)=(1, 1, 1)$, 
let 
$(C, h)$ 
be a Cantor metric space of 
type 
$(1, 1, 1)$;
otherwise, 
let 
$(C, h)$ 
be a Cantor metric space of totally exotic type 
$(u, v, w)$ 
(see Theorem \ref{thm:totexo}). 
We may assume that  
$\delta_{h}(C)\le 2$. 
Let 
$\zeta_{1}: [0, 1]\to [0, \infty)$ 
be a 
$\grdis((X, d), (Y, e))$-Lipschitz map 
with 
$\zeta_{1}^{{-1}}(0)=\{0, 1\}$
 (see Lemma \ref{lem:lip}). 
 For 
 $s\in (0, 1)$, 
put 
\[
(W, E_{s})=(R\times C, d_{s}\times_{\infty} (\zeta_{1}(s)\cdot  h)). 
\]
We define 
a map 
$f:[0, 1]\to \grsp$  
by 
\[
f(s)=
\begin{cases}
(X, d) & \text{if $s=0$;}\\
(Y, e) & \text{if $s=1$;}\\
(W, E_{s}) & \text{if $s\in (0, 1)$}
\end{cases}
\]
By Proposition \ref{prop:prodgeo}, 
the map
 $f$ 
 is a geodesic from 
$(X, d)$ 
to 
$(Y, e)$. 
Note that by the latter parts of Propositions 
\ref{prop:canonicalgeo} and \ref{prop:prodgeo},  
for all 
$s, t\in (0, 1)$
the trivial correspondence 
$\Delta_{W}$ 
is 
in 
$\relacc_{opt}(f(s), f(t))$. 

To construct geodesics continuously parametrized by 
the Hilbert cube, 
we use the spaces $(U(\mathcal{P}_{q}), e_{\mathcal{P}_{q}})$ for some $\mathcal{P}$ as  ``identifiers'' corresponding to 
$q\in  \qcube$ (see Proposition \ref{prop:Uiso}). 
In the following,  we define $\mathcal{P}=\{(P_{i, j}, p_{i, j})\}_{i\in \{1, 2, 3\}, j\in \cind}$,  depending on $(u, v, w)$,  to construct 
$(U(\mathcal{P}_{q}), e_{\mathcal{P}_{q}})$. 
In the case of 
$w=1$, 
we 
define 
\[
M=
\begin{cases}
\sup_{s\in [0, 1]}\dim_{A}(W, E_{s}) & \text{if $(u, v, w)=(1, 1, 1);$}\\
1 & \text{otherwise.}
\end{cases}
\]
We now show that $M$ is always finite. 
In the case of $(u, v, w)=(1, 1, 1)$, 
put
$(\mset, \msetdis_{s})=(X\times Y\times C, d_{s}\times_{\infty}(\zeta_{1}(s)\cdot h))$. 
Then, the space 
$(W, E_{s})$ is a subspace of 
$(\mset, \msetdis_{s})$ for all $s\in[0, 1]$. 
Thus, Lemma \ref{lem:subAssouad} implies 
$\dim_{A}(W, E_{s})\le \dim_{A}(\mset, \msetdis_{s})$.
Since the Assouad dimension is  
a
bi-Lipschitz invariant, 
and since $d_{s}$ is bi-Lipschitz equivalent  to 
$d\times_{\infty}e$, or $d$ or $e$, 
Lemma \ref{lem:finAssouad} implies 
$\dim_{A}(\mset, \msetdis_{s})\le \dim_{A}(X, d)+\dim_{A}(Y, e)+\dim_{A}(C, h)$
 for all $s\in [0, 1]$. 
Then, we have 
$\sup_{s\in [0, 1]}\dim_{A}(W, E_{s})\le \dim_{A}(X, d)+\dim_{A}(Y, e)+\dim_{A}(C, h)$, and hence 
$\sup_{s\in [0, 1]}\dim_{A}(W, E_{s})$ is finite. 
Namely, the number $M$ is always finite. 

In the case of $w=1$, 
for 
$i\in \{1, 2, 3\}$, 
and 
$j\in \cind$, 
let 
$c(i, j)\in (0, 1)$ 
satisfy 
\[
\dim_{A}(2^{\omega}, \beta_{c(i, j)})=M+i+2^{-j}. 
\]
Put 
$(P_{i, j}, p_{i, j})=(2^{\omega}, 2^{-j-1}l(\pi/6)\cdot \beta_{c(i, j)})$. 
Then, by Lemma \ref{lem:smprop}, 
\begin{enumerate}
\item each 
$(P_{i, j}, p_{i, j})$ 
is 
$(2^{M+6})$-doubling;
\item each 
$(P_{i, j}, p_{i, j})$ 
is an ultrametric space, 
and hence it is 
$\delta$-uniformly disconnected for all 
$\delta\in (0, 1)$;
\item 
each 
$(P_{i, j}, p_{i, j})$
 is 
 $(2^{-1/M})$-uniformly perfect.
\end{enumerate}

In the case of 
$w=0$, 
for 
$i\in \{1, 2, 3\}$, 
and 
$j\in \cind$, 
let 
$(P_{i, j}, p_{i, j})$ 
be the one-point metric space. 

In any case, 
put 
$\mathcal{P}=\{(P_{i, j}, p_{i, j})\}_{i\in \{1, 2, 3\}, j\in \cind}$. 
For each 
$q\in \qcube$,  
put 
$(U_{q}, e_{q})=(U(\mathcal{P}_{q}), e_{\mathcal{P}_{q}})$.
By Proposition \ref{prop:uprop}, 
we obtain 
$(U_{q}, e_{q})\in \mathscr{X}(1, 1, 1)$ if $w=1$; otherwise, 
$(U_{q}, e_{q})\in \mathscr{X}(1, 1, 0)$.   
Recall  that 
$U_{q}$
 contains the element represented as 
 $\infty$. 
Take 
$o\in W$. 
Let 
$\zeta_{2}: [0, 1]\to [0, \infty)$ 
be a 
$\grdis((X, d), (Y, e))$-Lipschitz map 
with 
$\zeta_{2}^{{-1}}(0)=A$ (see Lemma \ref{lem:lip}).  
For $s\in [0, 1]$ and $q\in \qcube$,  we put 
\[
Z_{q}=W\times \{\infty\}\cup \{o\}\times U_{q}, 
\]
and define a metric $H_{s, q}$ on $Z_{q}$ by 
\[
H_{s, q}=E_{s}\times_{\infty} (\zeta_{2}(s)\cdot e_{q}). 
\]
Note that if $s\in A\cap (0, 1)$, then 
$(Z_{q}, H_{s, q})$ is a pseudo-metric space, and 
there exists an isometry between 
$(W, E_{s})$ and the quotient metric space of  $(Z_{q}, H_{s, q})$ by 
the relation $x\sim y$ defined by  $H_{s, q}(x, y)=0$. 
The family  $\{(Z_{q}, H_{s, q})\}_{(s, q)\in (0, 1)\times \qcube}$ is considered as an 
`` expansion of a geodesic $\gamma$ equipped with the  identifiers
 $\{(U_{q}, e_{q})\}_{q\in \qcube}$.''

We define a map 
$F: [0, 1]\times \qcube \to \grsp$
by 
\[
F(s, q)=
\begin{cases}
f(s) & \text{if $s\in A$; }\\
(Z_{q}, H_{s, q}) & \text{if $s\in [0, 1]\setminus A$.}
\end{cases}
\]

We now prove that the map $F:  [0, 1]\times \qcube \to \grsp$ is an 
$A$-branching bunch of geodesics from $(X, d)$ to $(Y, e)$. 

By Proposition \ref{prop:Uconti}, 
we first observe  that the map 
$F$ is continuous.

By the definition of $F$, 
the conditions (\ref{item:end}) 
and 
(\ref{item:A}) in Definition \ref{df:bbg} are 
satisfied.

Since  for each 
$q\in \qcube$ 
we have 
$\delta_{e_{q}}(U_{q})\le 1$, 
and since 
the trivial correspondence 
$\Delta_{W}$ is 
in $\relacc_{opt}(W, E_{s}, W, E_{t})$, 
by  Proposition \ref{prop:tailedgeo}. 
the map 
$F_{q}:[0, 1]\to \grsp$ 
is a geodesic.
Thus the condition (\ref{item:Fq}) 
in Definition \ref{df:bbg}
is satisfied.

We now prove  the condition (\ref{item:noniso}) 
in Definition \ref{df:bbg}. 
Before doing that, 
depending on 
$(u, v, w)$, 
we  define isometrically invariant  operations picking out the 
space $(U_{q}, e_{q})$ from $F(s, q)$. 
Let 
$(S, m)$ 
be an arbitrary compact metric space. 
We denote by 
$\mathcal{C}_{M}(S, m)$ 
the 
set of all 
$x\in S$ 
for which there exists 
$r\in (0, \infty)$ 
such that for all 
$\epsilon\in (0, r)$
we have  
$\dim_{A}(B(x, \epsilon), m)\le M$.  
We denote by 
$\mathcal{I}(S, m)$ 
the set of all 
isolated points of 
$(S, m)$. 
Let 
$(u, v, w)\in \{0, 1\}$, 
we define 
\[
\mathcal{B}_{(u, v, w)}(S, m)=
\begin{cases}
\cl_{S}(S\setminus \mathcal{C}_{M}(S, m)) & \text{if $(u, v, w)=(1, 1, 1 )$;}\\
\cl_{S}(\mathcal{I}(S, m))  & \text{if $w=0$;}\\
\cl_{S}(S\setminus \mathcal{S}_{\mathscr{P}_{1}}(S, m)) & \text{if $u=0$ and $w$=1;}\\
\cl_{S}(S\setminus \mathcal{S}_{\mathscr{P}_{2}}(S, m)) & \text{if $(u, v, w)=(1, 0, 1)$,}
\end{cases}
\]
where 
$\cl_{S}$ 
is the closure operator of 
$S$. 
Then 
$\mathcal{B}_{(u, v, w)}(S, m)$ 
is an isometric invariant, i.e., 
if 
$(S, m)$ 
and 
$(S', m')$ 
are isometric, 
then so are 
$\mathcal{B}_{(u, v, w)}(S, m)$ 
and 
$\mathcal{B}_{(u, v, w)}(S', m')$. 
By the definitions of 
$M$, 
$\mathcal{P}$, 
and  
$Z_{q}$, 
and 
by  Lemma \ref{lem:timestotexo}, 
we see 
 that for every 
 $(s, q)\in ([0, 1]\setminus A)\times \qcube $ 
 the space  
 $\mathcal{B}_{(u, v, w)}(F(s, q))$ 
 is isometric to 
$(U_{q}, \zeta_{2}(s)\cdot e_{q})$. 

Take 
$(s, q), (t, r)\in ([0, 1]\setminus A)\times \qcube$. 
We now  prove that if 
$F(s, q)$ 
and 
$F(t, r)$ 
are isometric to each other, 
then 
$(s, q)= (t, r)$.  
Under this assumption, 
since 
$\mathcal{B}_{(u, v, w)}$ 
is invariant under isometric maps, 
$(U_{q}, \zeta_{2}(s)\cdot e_{q})$
 and 
 $(U_{r}, \zeta_{2}(t)\cdot e_{r})$ 
 are isometric. 
Thus, by Proposition \ref{prop:Uiso}, 
we obtain 
$q=r$.  
Since 
$F_{q}$ 
is a geodesic, we observe  that 
$s=t$. 
This implies the condition (\ref{item:noniso}) in 
Definition \ref{df:bbg}.

Subsequently, we 
 prove that 
for all 
$(s, q)\in [0, 1]\times \qcube$
we have 
$F(s, q)\in \mathscr{X}(u, v, w)$.  
By Lemma
\ref{lem:qsprod}, 
we first observe that 
$T_{\mathscr{P}_{1}}(W, E_{s})=u$ 
and 
$T_{\mathscr{P}_{2}}(W, E_{s})=v$. 
Then we also  observe  that 
$T_{\mathscr{P}_{1}}(W\times U_{q}, E_{t}\times_{\infty} (\zeta_{2}(s)e_{q}))=u$ 
and 
$T_{\mathscr{P}_{2}}(W\times U_{q}, E_{t}\times_{\infty} (\zeta_{2}(s)e_{q}))=v$. 
Since 
$F(s, q)$ 
is a subspace of 
$(W\times U_{q}, E_{t}\times_{\infty} (\zeta_{2}(s)e_{q})$, 
and since 
$F(s, q)$ 
contains 
$(W, E_{s})$, 
by Proposition \ref{prop:hered}, 
we obtain 
$T_{\mathscr{P}_{1}}(F(s, q))=u$ 
and 
$T_{\mathscr{P}_{2}}(F(s, q))=v$. 
By the definitions of 
$(U_{q}, e_{q})$ 
and 
 $F(s, q)$,  
 and 
 by Lemmas 
 \ref{lem:produp} 
 and 
 \ref{lem:unionup}, 
we infer that  
$T_{\mathscr{P}_{3}}(F(s, q))=w$.  
Thus,  we obtain 
$F(s, q)\in \mathscr{X}(u, v, w)$.

Therefore we conclude that 
 Theorem \ref{thm:q-emb} holds true. 
\end{proof}

\begin{cor}\label{cor:geodesicuvw}
For every 
$(u, v, w)\in \{0, 1, 2 \}^{3}$, 
the set 
$\mathscr{X}(u, v, w)$ 
is a geodesic space. 
\end{cor}

Next we prove Theorems \ref{thm:q-embinfinite}, 
\ref{thm:Cantoremb},  and 
\ref{thm:branch}. 
\begin{lem}\label{lem:perfectinfinite}
For a dimensional function 
$\dimdim$, 
there exists a perfect compact metric space 
$(X, d)$ 
such that 
$\dimdim(X, d)=\infty$. 
\end{lem}
\begin{proof}
Take a compact metric space 
$(Y, e)$ 
with 
$\dimdim(Y, e)=\infty$, 
and take  a perfect  metric space 
$(P, h)$. 
Then 
$(X\times P, e\times_{\infty}h)$ 
is a 
desired one. 
\end{proof}

\begin{proof}[Proof of Theorem \ref{thm:q-embinfinite}]
Let 
$(C, h)$ 
be a perfect compact metric space with 
$\dimdim(C, h)=\infty$. 
Let 
$(P_{i, j}, p_{i, j})$ 
be the one-point metric space. 
Put 
$\mathcal{P}=\{(P_{i, j}, p_{i, j})\}_{i\in \{1, 2, 3\}, j\in \cind }$. 
Then the proof can be presented  in  
the same way as
 the proof of 
Theorem \ref{thm:q-emb} 
in the case of 
$w=0$. 
\end{proof}

\begin{proof}[Proof of Theorem \ref{thm:Cantoremb}]
Let $(X, d), (Y, e)\in \ccc$.
We take a correspondence 
$R\in \relacc_{opt}(X, d, Y, e)$. 
For each 
$s\in (0, 1)$, 
put  $d_{s}=(1-s)\cdot d+s\cdot e$.
Depending on whether 
the metric spaces $\{(R, d_{s})\}_{s\in (0, 1)}$ are doubling or not, 
let $F: [0, 1]\times \qcube\to \grsp$ be the map constructed in the proof of Theorem \ref{thm:q-emb} in the case of $(u, v, w)=(1, 1, 1)$ or $(0, 1, 1)$, respectively. 
Since a perfect totally disconnected  compact metric space  is in $\ccc$, 
we have $F(s, q)\in \ccc$ for all $(s, q)\in [0, 1]\times\qcube$. This completes the proof. 
\end{proof}

\begin{proof}[Proof of Theorem \ref{thm:branch}]
Since 
$[0, 1]$ 
is separable 
and 
$\card(\grsp)=2^{\aleph_{0}}$, 
the cardinality of geodesics is at most the continuum. 
By Theorem \ref{thm:q-emb}, 
we conclude that the cardinality of geodesics is exactly 
continuum. 
\end{proof}

\begin{cor}
For all dimensional  function 
$\dimdim$, 
the set 
$\iii(\dimdim)$ 
is a  geodesic space, and 
the set $\ccc$ is a geodesic space. 
\end{cor}

We next construct a topological embedding 
of 
$\qcube$ 
into 
$\grsp$ 
 using 
a branching bunch of geodesics stated in 
Theorem \ref{thm:q-emb}.

A metrizable space 
$X$ 
is said to be 
 an \emph{absolute retract} (\emph{AR} for short), 
 if for every topological embedding 
 $f$ 
 from 
 $X$ 
 into a metrizable space 
 $Y$ 
 the image 
 $f(X)$ 
 is a retract of 
 $Y$. 
The following is known as the 
Keller--Dobrowolski--Toru\'nczyk  characterization of 
the Hilbert cube 
$\qcube$. 
The proof is presented  in 
\cite[Theorem 3.9.2]{Sakai2020topology}. 
\begin{thm}\label{thm:kdt}
Let 
$E$ 
be a metrizable 
linear space. 
A convex set  of  
$E$ is 
homeomorphic to the Hilbert cube 
$\qcube$ 
if and only if 
it is an infinite-dimensional compact AR.
\end{thm}
\begin{cor}\label{cor:Hilbertcube}
Define 
$m: [0, 1]\to [0, 1]$ 
by 
$m(s)=\min\{s, 1-s\}$, 
and 
define 
$\mathbf{K}=\{\, (s, m(s)\cdot q)\mid s\in [0, 1], q\in \qcube\, \}$. 
Then 
$\mathbf{K}$ 
is homeomorphic to the Hilbert cube
 $\qcube$. 
\end{cor}
\begin{proof}
The set 
$\mathbf{K}$ 
is a  compact convex infinite-dimensional subset of the 
locally convex topological linear space 
$\rr^{\nn}$. 
By the Dugundji extension 
theorem (see \cite[Theorem 1.13.1]{Sakai2020topology}), 
$\mathbf{K}$ 
is an AR. 
Thus, 
by Theorem \ref{thm:kdt},
the set 
$\mathbf{K}$ 
is homeomorphic to 
$\qcube$. 
\end{proof}

\begin{proof}[Proof of Theorem \ref{thm:topemb}]
Let 
$(X, d), (Y, e)\in \mathscr{S}$. 
Let 
$F:[0, 1]\times \qcube\to \mathscr{X}(u, v, w)$ 
be 
a 
$\{0, 1\}$-branching bunch of  geodesics  stated 
in Theorem  \ref{thm:q-emb} 
or 
\ref{thm:q-embinfinite}. 
We define a map  
$G: \mathbf{K}\to \mathscr{X}(u, v, w)$ 
by 
\[
G(a, b)=
\begin{cases}
(X, d) & \text{if $a=0$;}\\
(Y, e) & \text{if $a=1$;}\\ 
F(a, b/m(a)) & \text{otherwise.}
\end{cases}
\] 
Subsequently, 
$G$ 
is continuous and injective.
Since 
$\mathbf{K}$ 
is compact, 
the map 
$G$ 
is a topological embedding. 
This completes  the proof. 
 \end{proof}

\begin{proof}[Proof of Theorem \ref{thm:locinfinite}]
Let 
$\mathscr{S}$ 
be any one of  
$\mathscr{X}(u, v, w)$ 
for some triple 
$(u, v, w)\in \{0, 1, 2\}^{3}$ 
or 
$\iii(\dimdim)$ 
for some 
dimensional function 
$\dimdim$
or $\ccc$. 
Let $O$ be a non-empty open subset of $\grsp$. 
By Lemma \ref{lem:denseuvw}, 
the set $\mathscr{S}$ is dense in 
$\grsp$. Thus, 
$\mathscr{S}\cap O$ is non-empty. 
Take $(X, d)\in \mathscr{S}\cap O$
and take $r\in (0, \infty)$ such that 
the closed ball centered at $(X, d)$ with radius $r$
is contained in $O$. 
Take $(Y, e)\in \mathscr{S}$ with 
$\grdis((X, d), (Y, e))=r$.
The existence of such $(Y, e)$ is 
guaranteed by the existence of geodesics in 
$\mathscr{S}$.  
Applying 
Theorems \ref{thm:q-emb}
 and  \ref{thm:topemb} to $(X, d)$, 
 $(Y, e)$ and $A=\{0, 1\}$, 
 we obtain a $\{0, 1\}$-branching bunch $F: [0, 1]\times \qcube\to \mathscr{S}$ of geodesics from 
 $(X, d)$ to $(Y, e)$. 
 By the choice of $r$, 
 we obtain  $F([0, 1]\times \qcube)\subset \mathscr{S}\cap O$. 
 Since $F([0, 1]\times \qcube)$ contains a homeomorphic copy 
 of $\qcube$, 
 the set $\mathscr{S}\cap O$ is infinite-dimensional. 
\end{proof}

\section{Table of symbols}\label{sec:table}
\setlongtables
\begin{tabularx}{\textwidth}{|l|X|}

\endhead
\hline 

Symbol 
&
Description
\\
\hline
\hline 

$\grsp$ &
The set of all isometry classes of compact metric spaces. 
\\
\hline

$\grdis$ &
The Gromov--Hausdorff distance. 
\\
\hline

$\qcube$ & 
The Hilbert cube. 
\\ 
\hline

$B(x, r)$ & 
The closed ball centered at $x$ with radius $r$. 
\\
\hline

$\cind$ &
The set of all positive integers. 
\\
\hline 

$\met(X)$ & 
The set of all topologically compatible metrics of 
a metrizable space $X$.
\\
\hline

$\mathscr{P}_{k}$ & 
$\mathscr{P}_1$: the doubling property, 
 $\mathscr{P}_2$: the uniform disconnectedness, 
  $\mathscr{P}_3$: the uniform perfectness. 
  \\
  \hline

$\setd$ &
The set of all doubling compact metric spaces in $\grsp$. 
\\
\hline 

$\setud$ & 
The set of all uniformly disconnected compact spaces in $\grsp$. 
\\
\hline

$\setup$ & 
The set of all uniformly perfect compact spaces in $\grsp$. 
\\
\hline 

$T_{\mathscr{P}}(X, d)$ &
The truth value of the statement that $(X, d)$ satisfies a property $\mathscr{P}$. 
\\
\hline 

$\mathscr{Q}_{k}$ & 
$\mathscr{Q}_{1}=\setd$, 
$\mathscr{Q}_{2}=\setud$, 
and $\mathscr{Q}_{3}=\setup$. 
\\
\hline

$\mathscr{E}_{k}(u)$ & 
$\mathscr{E}_{k}(0)=\grsp \setminus \mathscr{Q}_{k}$, 
$\mathscr{E}_{k}(1)=\mathscr{Q}_{k}$, 
and 
$\mathscr{E}_{k}(2)=\grsp$. 
\\
\hline

$\mathscr{X}(u, v, w)$
& 
$\mathscr{X}(u, v, w)=\mathscr{E}_{1}(u)\cap 
\mathscr{E}_{2}(v)\cap 
\mathscr{E}_{3}(w)$. 
\\
\hline 

$\iii(\dimdim)$ & 
The set of all $(X, d)$ in $\grsp$ with 
$\dimdim(X, d)=\infty$ for a dimensional function $\dimdim$. 
\\
\hline 

$\ccc$ &
The set of all compact metric spaces homeomorphic to 
the (middle-third) Cantor set. 
\\
\hline 

$\lor$ & 
The maximum operator on $\rr$. 
\\
\hline 

$\land$ & 
The minimum operator on $\rr$. 
\\
\hline

$\card(A)$ & 
The cardinality of a set $A$. 
\\
\hline

$\delta_{d}(A)$ & 
The diameter of $A$ by a metric $d$.  
\\
\hline 

$\alpha_{d}(A)$ &
$\alpha_{d}(A)=
\inf\{\, d(x, y)\mid x\neq y, x, y\in A\, \}$. 
\\
\hline 

$\dim_{A}(X, d)$ & 
The Assouad dimension of 
$(X, d)$. 
\\
\hline 

$d\times_{\infty} e$ & 
The product metric defined by 
$(d\times_{\infty} e)((x, y), (u, v))=
d(x, u) \lor e(y, v)$. 
\\
\hline 

$X\sqcup Y$ & 
The direct sum of $X$ and $Y$. 
\\
\hline 

$\coprod_{i\in I}X_{i}$ & 
The direct sum of a family $\{X_{i}\}_{i\in I}$.
\\
\hline 

$\mathcal{S}_{\mathscr{P}}(X,d)$  & 
The set of all points in 
$(X, d)$ 
of which no neighborhoods 
satisfy a property 
$\mathscr{P}$. 
\\
\hline 

$\rela(X, Y)$ & 
The set of all correspondences 
of 
$X$ 
and  
$Y$ (see Subsection 
\ref{subsec:gh}). 
\\
\hline 

$\relacc(X, d, Y, e)$ & 
The set of all 
closed correspondences in 
$X\times Y$ in 
$\rela(X, d, Y, e)$. 
\\
\hline

$\relacc_{opt}(X, d, Y, e)$ & 
The set of all optimal 
$R\in \relacc(X, d, Y, e)$. 
\\
\hline

$\dis(R)$ & 
$\dis(R)=\sup_{(x,y), (u,v)\in R}|d(x,u)-e(y,v)|$.
\\
\hline

$\Delta_{X}$ & 
$\Delta_{X}=\{\, (x, x)\mid x\in X\, \}\in \rela(X, X)$. 
\\
\hline 

$\widehat{n}$ & 
$\widehat{n}=\{1, \dots, n\}$.
\\
\hline

$(T(\mathcal{X}), d_{\mathcal{X}})$ & 
The telescope space for $\mathcal{X}$ 
defined in Section \ref{sec:const}.
\\
\hline 

$\mathcal{P}_{q}$ & 
$\mathcal{P}_{q}=(\mathcal{P}, q)$,
where 
$q\in \qcube$ and 
$\mathcal{P}=
\{(P_{i, j}, p_{i, j})\}_{i\in \{1, 2, 3\}, j\in \cind}$
be a sequence of compact metric spaces  indexed by 
$\{1, 2, 3\}\times \cind$ (see Definition  \ref{df:tail}). 
\\
\hline 

$(U(\mathcal{P}_{q}), e_{\mathcal{P}_q})$ & 
The telescope space constructed 
in Definition \ref{df:tail}. 
\\
\hline

$2^{\omega}$ & The set of  all maps from $\zz_{\ge 0}$
into $\{0, 1\}$. 
\\
\hline

$v(x, y)$  
&
$v(x, y)=\min\{\, n\in \zz_{\ge 0}\mid x(n)\neq y(n)\, \}$ if 
$x\neq y$; 
otherwise 
$v(x, y)=\infty$. 
\\
\hline

$g_{c}:\zz_{\ge 0}\cup\{\infty\}\to [0, \infty)$ & 
$g_{c}(n)=c^{n}$ 
if 
$n\in \zz_{\ge 0}$ 
and 
$g_{c}(\infty)=0$ for $c\in (0, 1)$.  
\\
\hline 

$\beta_{c}$
&
The metric $\beta_{c}$ 
on 
$2^{\omega}$ 
defined 
by 
$\beta_{c}(x, y)=g_{c}(v(x, y))$
for $c\in (0, 1)$.
\\
\hline 

$\aaa(v, X, d)$ & 
The set of all 
$x\in X$ 
for which 
there exists 
$r\in (0, \infty)$ 
such that 
for every 
$\epsilon \in (0, r)$ 
we have 
$\dim_{A}(B(x, \epsilon), d)$. 
The auxiliary invariant under isometries  for the proof of Proposition \ref{prop:Uiso}. 
\\
\hline 

\end{tabularx}

\begin{ac}
The author would like to thank Takumi Yokota  
for 
raising  questions,  
for the many stimulating conversations, and 
for the helpful comments. 
The author  would also  like to thank Editage (www.editage.com) for English language editing.
The author is deeply grateful 
to  the referee for helpful comments and suggestions.
\end{ac}

\begin{conflict}
The author states no conflict of interest. 
\end{conflict}

\bibliographystyle{amsplain}
\bibliography{bibtex/infiH.bib}

\end{document}